\documentclass[a4paper, 10pt, oneside, onecolumn,sort&compress]{elsarticle}
\usepackage{amssymb}
\usepackage{mathrsfs}
\usepackage{amsfonts}
\usepackage[title]{appendix}
\usepackage{syntonly}
\newcommand{\lap}{\mbox{$\bigtriangleup$}}
\newcommand{\ra}{{\mbox{$\rightarrow$}}}
\newtheorem{thm}{Theorem}[section]

\newdefinition{remark}{Remark}[section]
\newdefinition{example}{Example}[section]
\newdefinition{corollary}{Corollary}[section]
\newdefinition{proposition}{Proposition}[section]
\newdefinition{definition}{Definition}[section]
\newproof{proof}{Proof}
\newproof{pot}{Proof of Theorem \ref{the1.1}}
\usepackage{amssymb,amsmath}

\numberwithin{equation}{section}
\journal{Nonlinear Analysis: Theory, Methods $\&$ Applications}

\begin{document}
\begin{frontmatter}

\title{\textbf{A Liouville Theorem for a Class of Fractional Systems in $\mathbb{R}^n_+$}}

\author[mymainaddress]{Lizhi Zhang}
\ead{zhanglizhi2016@mail.nwpu.edu.cn}

\author[mymainaddress]{Mei Yu\corref{mycorrespondingauthor}}
\cortext[mycorrespondingauthor]{Corresponding author.}
\ead{yumei301796@gmail.com}

\author[mymainaddress1]{Jianming He}
\ead{jmingwivi@163.com}

\address[mymainaddress]{Department of Applied Mathematics, Northwestern Polytechnical University, Xi'an, Shannxi, 710129, P. R. China}

\address[mymainaddress1]{Institute of Technology and Standards, China Academy of Information and Communications Technology, Beijing, 100191, P. R. China}

\begin{abstract}
Let $0<\alpha,\beta<2$ be any real number. In this paper, we investigate a class of fractional elliptic problems of the form
\begin{equation*}
\left\{\begin{array}{lll}
(-\lap)^{\alpha/2} u(x)=f(v(x)), &  \\
(-\lap)^{\beta/2} v(x)=g(u(x)), & \qquad x\in\mathbb{R}^n_+,\\
u,v\geq0, & \qquad x\not\in\mathbb{R}^n_+.
\end{array}\right.
\end{equation*}
Applying the iteration method and the direct method of moving planes for the fractional Laplacian, without any decay assumption on the solutions at infinity, we prove the Liouville theorem of nonnegative solutions under some natural conditions on $f$ and $g$.
\end{abstract}


\end{frontmatter}


\section{Introduction}
\label{Section 1}
Symmetry, Liouville theorems and nonexistence are very useful in studying semi-linear elliptic equations and systems. For example, in \cite{APP}, \cite{FLN}, \cite{GS1} and \cite{HL}, those properties played an essential role in deriving a priori bounds for solutions; and in \cite{CC}, \cite{KK}, \cite{MK}, \cite{MS} and \cite{PS}, they were used to obtain uniqueness of solutions. There are many other applications.

In this paper, we employ a new idea, i.e. the iteration method, to establish such properties for the following fractional system with different orders:
\begin{equation}
\left\{\begin{array}{lll}
(-\lap)^{\alpha/2} u(x)=f(v(x)), &  \\
(-\lap)^{\beta/2} v(x)=g(u(x)), & \qquad x\in\mathbb{R}^n_+,\\
u,v\equiv0, & \qquad x\not\in\mathbb{R}^n_+,
\end{array}\right.
\label{0}
\end{equation}
where $0<\alpha,\beta<2$, $n\geq3$, and $\mathbb{R}^n_+:=\{x=(x_1,x_2,\cdots,x_n)\in\mathbb{R}^n|x_n>0\}$ is the upper half Euclidean space.

Let us begin with the definition of the fractional Laplacian. The fractional Laplacian in $\mathbb{R}^n$ is a nonlocal pseudo-differential operator, assuming the form
\begin{equation}
(-\lap)^{\alpha/2}u(x)=C_{n,\alpha}PV\int_{\mathbb{R}^n}\frac{u(x)-u(z)}{|x-z|^{n+\alpha}}dz,
\label{pp}
\end{equation}
where $0<\alpha<2$ is any real number and $PV$ stands for the Cauchy principle value.

One can also extend the operator in (\ref{pp}) to a wider space of functions
\begin{equation}
{\cal L}_{\alpha}(\mathbb{R}^n)=\{u\left|\right.\int_{\mathbb{R}^n}\frac{|u(x)|}{1+|x|^{n+\alpha}}dx<\infty\}
\label{L}
\end{equation}
by
$$
<(-\lap)^{\alpha/2}u, \phi>=\int_{\mathbb{R}^n}u(-\lap)^{\alpha/2}\phi dx,\quad \text{for all} ~\phi\in C_0^\infty(\mathbb{R}^n).
$$
It is easy to verify that for $u\in C_{loc}^{1,1}(\mathbb{R}^n)\cap {\cal L}_\alpha(\mathbb{R}^n)$, the integral on the right hand side of (\ref{pp}) is well defined.

There are several distinctly different ways to define the fractional Laplacian in a domain $\Omega\in\mathbb{R}^n$, which coincide when the domain is the entire Euclidean space, but can otherwise be quite different. For example, Cabre and Tan \cite{CT} studied a nonlocal problem, taking as the fractional Laplacian the operator with the same eigenfunctions as the regular Laplacian, by extending to one further dimension. Another way is to restrict the integration to the domain:
\begin{equation}
(-\lap)_\Omega^{\alpha/2}u(x)=C_{n,\alpha}PV\int_{\Omega}\frac{u(x)-u(z)}{|x-z|^{n+\alpha}}dz,
\label{_+}
\end{equation}
known as the regional fractional Laplacian \cite{Gu}. To guarantee the validity of the integration on the right hind side of (\ref{_+}), one need $u\in C_{loc}^{1,1}(\Omega)\cap{\cal L}_\alpha(\Omega)$, where ${\cal L}_\alpha(\Omega)$ is defined as (\ref{L}) by substituting $\mathbb{R}^n$ for $\Omega$.

In recent years, the fractional Laplacian has been frequently used to model diverse physical phenomena, such as phase transitions, flame propagation, the turbulence, water waves, anomalous diffusion and quasi-geostrophic flows (see \cite{BoG} \cite{Co} \cite{CaV} \cite{TZ} and the references therein). It also has various applications in probability, optimization and finance (see \cite{A} \cite{Be} \cite{CT}). In particular, the fractional Laplacian can be understood as the infinitestmal generator of a stable L\'{e}vy process \cite{Be}.

In our work, we consider the fractional Laplacians $(-\lap)^{\alpha/2}$, $(-\lap)^{\beta/2}$ in the following setting
\begin{equation}
(-\lap)^{\alpha/2}u(x)=C_{n,\alpha}PV\int_{\mathbb{R}^n_+}
\frac{u(x)-u(z)}{|x-z|^{n+\alpha}}dz,~~(-\lap)^{\beta/2}v(x)=
C_{n,\beta}PV\int_{\mathbb{R}^n_+}\frac{v(x)-v(z)}{|x-z|^{n+\beta}}dz,
\label{_+_}
\end{equation}
and we suppose that $u\in C_{loc}^{1,1}(\mathbb{R}^n_+)\cap {\cal L}_\alpha(\mathbb{R}^n_+)$, $v\in C_{loc}^{1,1}(\mathbb{R}^n_+)\cap {\cal L}_\beta(\mathbb{R}^n_+)$.

We first list several \textsl{maximum principles} for System (\ref{0}) in Section 2, with these maximum principles, using \textsl{the iteration method and the direct method of moving planes for the fractional Laplacian}, we mainly prove
\begin{thm}
Assume that $f(t),g(t)\geq0$ are strictly increasing about $t$ in $[0,+\infty)$, and $\frac{f(t)}{t^{p_0}}$,$\frac{g(t)}{t^{q_0}}$ are non-increasing in $t>0$ with $p_0=\frac{n+\alpha}{n-\beta}$, $q_0=\frac{n+\beta}{n-\alpha}$. Suppose the nonnegative solutions of (\ref{0}) $u\in\left({\cal L}_\alpha(\mathbb{R}^n_+)\cap C_{loc}^{1,1}(\mathbb{R}^n_+)\cap C(\mathbb{R}^n)\right)$, $v\in\left({\cal L}_\beta(\mathbb{R}^n_+)\cap C_{loc}^{1,1}(\mathbb{R}^n_+)\cap C(\mathbb{R}^n)\right)$, then $u=v\equiv0$ in $\mathbb{R}^n$, and $f(0)=g(0)=0$.
\label{thm1}
\end{thm}

\begin{remark}
If $f(0)$ or $g(0)$ does not equal to $0$, then (\ref{0}) admits no nonnegative solutions under the conditions of Theorem \ref{thm1}.
\end{remark}

Everyone knows that a majority of results are about the fractional systems involving the same order operators, please see \cite{14}, \cite{33}, \cite{34} and \cite{MLZ}. However, few results have been derived for fractional systems with different orders. Here what we want to emphasize is that, our result is precisely about a class of systems with different orders $\alpha,\beta$ but also contains the same-order case $\alpha=\beta$; moreover, we allow $\alpha,\beta$ to be any real number between 0 and 2, and as far as we know no other results have managed this, even for systems with more simpler nonlinear terms such as
\begin{equation}
\left\{\begin{array}{lll}
(-\lap)^{\alpha/2} u(x)=v^p(x), &  \\
(-\lap)^{\beta/2} v(x)=u^q(x), & \qquad x\in\mathbb{R}^n_+,\\
u,v\equiv0, & \qquad x\not\in\mathbb{R}^n_+,
\end{array}\right.
\label{0_0}
\end{equation}
with $p\neq q$. That is

\begin{remark}
In Theorem \ref{thm1}, the orders $\alpha,\beta$ of the fractional system (\ref{0}) can be any real number between 0 and 2, no matter $\alpha=\beta$ or $\alpha\neq\beta$.
\label{rem2}
\end{remark}

The significance of our work also lies in the generalities of the nonlinear terms $f(t),g(t)$, $t\geq0$. Obviously, $f(t),g(t)$ represent a large family of functions such as $\ln(1+t)$, $at^p$, $bt+\ln(1+t)$ and so on. When $f(t)=t^p$, $g(t)=t^q$ and $\alpha=\beta$, the system (\ref{0}) becomes the famous fractional Lane-Emden system in $\mathbb{R}^n_+$
\begin{equation}
\left\{\begin{array}{lll}
(-\lap)^{\alpha/2} u(x)=v^p(x), &  \\
(-\lap)^{\alpha/2} v(x)=u^q(x), & \qquad x\in\mathbb{R}^n_+,\\
u,v\equiv0, & \qquad x\not\in\mathbb{R}^n_+,
\end{array}\right.
\label{00}
\end{equation}
where $0<\alpha<2$.

In the past several decades, the celebrate Lane-Emden system
\begin{equation}
\left\{\begin{array}{lll}
-\lap u(x)=v^p(x), &  \\
-\lap v(x)=u^q(x), & \quad x\in\mathbb{R}^n.\\
u,v>0, &
\end{array}\right.
\label{i01}
\end{equation}
has played a central role in the progression of nonlinear analysis. Some basic results, for example, the eigenfunction theory, the critical point theory, the a priori estimates and the Liouville-type theorems, have been obtained, among which the Liouville theorems have grasped more and more attentions, but have not yet been fully studied.

The famous Lane-Emden conjecture states

\textsl{For $p,q>0$, problem (\ref{i01}) possesses no classical solutions in the subcritical case $\frac{1}{p+1}+\frac{1}{q+1}>1-\frac{2}{n}$.}

Proving such a nonexistence result seems to be challenging, and it is still open for dimensions $n\geq5$. In an important paper \cite{6}, Mitidiery authenticated the Lane-Emden conjecture for radial solutions, and proved that (\ref{i01}) has bounded radial classical solutions in the critical and the supercritical cases ($\frac{1}{p+1}+\frac{1}{q+1}\leq1-\frac{2}{n}$), which means that the nonexistence theorem is optimal for radial solutions. Serrin and Zou \cite{77} also derived the above existence result. For non-radial solutions, Souto \cite{9}, Mitidieri \cite{6} and Serrin and Zou \cite{7} established the conjecture in dimensions $n=1,2$, while in $\mathbb{R}^3$, Serrin and Zou \cite{7} proved the nonexistence of polynomially bounded solutions, an assumption that was dropped by Pol\'{a}\u{c}ik, Quittner and Souplet \cite{8} a decade later. More recently, Souplet \cite{9} settled the conjecture completely in $\mathbb{R}^4$ and partly in higher dimensions $n\geq5$. Further evidence supporting the conjecture can also be found in \cite{12}, \cite{34}, \cite{11}, \cite{10} and \cite{13}.

Along with the emerging of the fractional Laplacian, the counterpart of the Lane-Emden system involving the fractional Laplacian, i.e. the fractional Lane-Emden system, is also getting more and more attentions, but is much less understood than the Lane-Emden system. The main difficulty is caused by the non-locality of the fractional Laplacian. To overcome it, Chen, Li and Ou \cite{ChLO} introduced \textsl{the method of moving planes in integral forms}. Latter, Chen, Li and Li \cite{CLL} developed \textsl{a direct method of moving planes for the fractional Laplacian}, which enable one to deal with the fraction equations and systems directly.

Recently, Leite and Montenegro \cite{30} established the existence and uniqueness of positive viscosity solutions to the fractional Lane-Emden system in $\Omega$
\begin{equation}
\left\{\begin{array}{lll}
(-\lap)^{\alpha/2}u(x)=v^p(x), &\\
(-\lap)^{\alpha/2}v(x)=u^q(x), &x\in\Omega,\\
u=v=0, &x\not\in\Omega,
\end{array}\right.
\label{i04}
\end{equation}
with $pq\neq1,~p,q>0$ in the supercritical case $\frac{1}{p+1}+\frac{1}{q+1}>\frac{n+\alpha}{n}$, where $\Omega\subseteq\mathbb{R}^n$ is a smooth bounded domain. Quaas and Xia \cite{32} considered (\ref{00}) and proved nonexistence of positive viscosity solutions under $1<p,q<\frac{n+2\alpha}{n-2\alpha}$.

Actually, we have also partially solved the fractional Lane-Emden conjecture, no matter the same-order one or the different-order one, in dimensions $n\geq3$ as a byproduct of Theorem \ref{thm1}:
\begin{corollary}
From Theorem \ref{thm1} and Remark \ref{rem2}, one can see that
\begin{description}
\item[~~~\emph{(\textbf{i})}] for any $0<\alpha<2$ and $0<p,q\leq\frac{n+\alpha}{n-\alpha}$, we also obtained the nonexistence of positive solutions to the fractional Lame-Emden system (\ref{00}) in $\mathbb{R}^n_+$ for dimensions $n\geq3$;
\item[~~~\emph{(\textbf{ii})}] for any $0<\alpha,\beta<2$ and $0<p\leq\frac{n+\alpha}{n-\beta}$, $0<q\leq\frac{n+\beta}{n-\alpha}$, we also obtained the nonexistence of positive solutions to the fractional Lame-Emden system (\ref{0_0}) in $\mathbb{R}^n_+$ for dimensions $n\geq3$.
\end{description}
\label{rem3}
\end{corollary}

In the proof of Theorem \ref{thm1}, some new ideas are involved. Here we briefly illustrate them. During the proof, we first derive that either the solutions $u(x)=v(x)\equiv0$ or $u(x),v(x)>0$ in $\mathbb{R}^n_+$ by virtue of (\ref{0}) and Theorem \ref{thm2.1}. For the case $u(x),v(x)>0$ in $\mathbb{R}^n_+$, to apply the method of moving planes, we make proper Kelvin transforms centered at $x^0\in\partial\mathbb{R}^n_+$ ($\partial\mathbb{R}^n_+:=\{x=(x_1,x_2,\cdots,x_n)\in\mathbb{R}^n|x_n=0\}$ is the boundary of $\mathbb{R}^n_+$, and $x^0$ is arbitrarily chosen):
$$
\bar{u}(x)=\frac{1}{|x-x^0|^{n-\alpha}}u(x^0+\frac{x-x^0}{|x-x^0|^2}),~~x\in\mathbb{R}^n_+,
$$
$$
\bar{v}(x)=\frac{1}{|x-x^0|^{n-\beta}}v(x^0+\frac{x-x^0}{|x-x^0|^2}),~~x\in\mathbb{R}^n_+,
$$
then $\bar{u}, ~\bar{v}$ satisfy
\begin{equation}
(-\lap)^{\alpha/2} \bar{u}(x)=\frac{1}{|x-x^0|^{n+\alpha}}f(|x-x^0|^{n-\beta}\bar{v}(x)),\quad x\in\mathbb{R}^n_+,
\label{iyy03002}
\end{equation}
\begin{equation}
(-\lap)^{\beta/2} \bar{v}(x)=\frac{1}{|x-x^0|^{n+\beta}}g(|x-x^0|^{n-\alpha}\bar{u}(x)),\quad x\in\mathbb{R}^n_+,
\label{iyy3002}
\end{equation}
respectively. Now through the moving plane method (moving the planes along any direction which is perpendicular to the $x_n$ axis), we have two possibilities: ($i$) $\bar{u}$, $\bar{v}$ are symmetric about some line $l_Q\neq L_{x^0}$, ($ii$) $\bar{u}$, $\bar{v}$ are symmetric about $l_{x^0}$, where $l_Q$ ($l_{x^0}$) denotes the line parallel to $x_n$-axis and passing through Point $Q$ ($x^0$).

Here the new ideas we want to underline are that, in possibility ($i$), we derive by (\ref{iyy03002}), (\ref{iyy3002}) and the symmetry of $\bar{u}$, $\bar{v}$ that
\begin{equation}
u(x)\sim(\frac{1}{|x|^{n-\alpha}}),~~v(x)\sim(\frac{1}{|x|^{n-\beta}})~~\text{at ~infinity},
\label{00.00}
\end{equation}
$$
f(t)=C_1t^{\frac{n+\alpha}{n-\beta}} ~\text{on}~ (0,max_{\mathbb{R}^n_+}v],~~g(t)=C_2t^{\frac{n+\beta}{n-\alpha}} ~\text{on}~ (0,max_{\mathbb{R}^n_+}u]
$$
for some constants $C_1,C_2>0$, and hence reduce (\ref{0}) into
\begin{equation*}
\left\{\begin{array}{lll}
(-\lap)^{\alpha/2} u(x)=C_1v^{\frac{n+\alpha}{n-\beta}}, &  \\
(-\lap)^{\beta/2} v(x)=C_2u^{\frac{n+\beta}{n-\alpha}}, & \qquad x\in\mathbb{R}^n_+,\\
u,v\equiv0, & \qquad x\not\in\mathbb{R}^n_+,
\end{array}\right.
\end{equation*}
which enable us to apply the method of moving planes in integral forms directly to $u(x),v(x)$ along the positive $x_n$ direction, and finally obtain that $u,v$ are increasing about $x_n$. This is a contradiction with (\ref{00.00}), therefore (\ref{0}) admits no positive solutions.

To ($ii$), it immediately follows that $u=u(x_n),v=v(x_n)$, and this is a contradiction with the finiteness of $u,v$ respectively, which we can arrive at through the iteration method and ingenious computation.

This paper is organized as follows. In section 2, we list several maximum principles which will be used in the proof of Theorem \ref{thm1}. Section 3 is mainly dedicated to the proof of Theorem \ref{thm1}. In Section 4, we prove four claims which are used in Section 3.

Throughout this paper, we denote $c,~c_0,~A,~C,~A_i,~B_i,~C_i,~i\in N$ as positive constants whose values may be different in different lines.

\section{Three Known Maximum Principles}

Similarly to Theorem 2.1 in \cite{CLL}, we obtained the following maximum principle for $\alpha$-super-harmonic functions, where $(-\lap)^{\alpha/2}$ is defined as (\ref{_+_}).
\begin{thm}(Maximum Principle) Let $\Omega$ be an bounded domain in $\mathbb{R}^n_+$. Assume that $u\in C^{1,1}_{loc}(\Omega)\cap\cal{L}_\alpha$$(\mathbb{R}^n_+)$ is semi-continuous on $\bar{\Omega}$ for $0<\alpha<2$ and satisfies
\begin{equation}
\left\{\begin{array}{ll}
(-\lap)^{\alpha/2}u(x)\geq0, & \qquad in ~\Omega, \\
u(x)\geq0, & \qquad in ~\mathbb{R}^n_+\backslash\Omega,
\end{array}\right.
\label{thm1.131}
\end{equation}
then
\begin{description}
\item[~~~\emph{(i)}]
\begin{equation}
u(x)\geq0~ ~in ~\Omega;
\label{thm1.01}
\end{equation}
\item[~~~\emph{(ii)}] If $u=0$ somewhere in $\Omega$, then
\begin{equation*}
u(x)=0~almost~everywhere~in~\mathbb{R}^n_+;
\label{eq1.01}
\end{equation*}
\item[~~~\emph{(iii)}] Conclusions (i) and (ii) hold for the unbounded region $\Omega$ if we further assume that
\begin{equation}
\displaystyle\underset{|x| \ra \infty}{\underline{\lim}}u(x)\geq0.
\label{eq1.02}
\end{equation}
\end{description}
\label{thm2.1}
\end{thm}

The proof of Theorem \ref{thm2.1} here is similar to the proof of Theorem 2.1 in \cite{CLL}, now let us prove it briefly.

\begin{proof} If (\ref{thm1.01}) does not hold, then the semi-continuity of $u$ on $\bar{\Omega}$ indicates that there exists a $x^0\in\bar{\Omega}$ such that
$$
u(x^0)=\min_{\bar{\Omega}}u<0.
$$
From (\ref{thm1.131}) one knows that $x^0$ is in the interior of $\Omega$, then it follows that
\begin{eqnarray}
(-\lap)^{\alpha/2}u(x^0)&=&C_{n,\alpha}PV\int_{\mathbb{R}^n_+}
\frac{u(x^0)-u(y)}{|x^0-y|^{n+\alpha}}dy
\nonumber\\
&\leq&C_{n,\alpha}\int_{\mathbb{R}^n_+\setminus\Omega}
\frac{u(x^0)-u(y)}{|x^0-y|^{n+\alpha}}dy
\nonumber\\
&<&0,
\label{aayyy2}
\end{eqnarray}
which contradicts with inequality (\ref{thm1.131}). This verifies (\ref{thm1.01}).

If $u(x^o)=0$ at some point $x^o\in\Omega$, then by
$$
0\leq(-\lap)^{\alpha/2}u(x^o)=C_{n,\alpha}PV\int_{\mathbb{R}^n_+}
\frac{-u(y)}{|x^o-y|^{n+\alpha}}dy
$$
and $u\geq0$, we must have
$$
u(x)=0~~\text{almost everywhere in}~\mathbb{R}^n_+.
$$

If (\ref{eq1.02}) holds, we can still get that the minimum point of $u$ is in the interior of $\Omega$. Thus, similarly to the bounded case, we can immediately obtain that conclusions \textbf{(i)} and \textbf{(ii)} are still true for unbounded $\Omega$.

This completes the proof.

\end{proof}

Let $T_\lambda$ be a hyperplane in $\mathbb{R}^n_+$, choose any direction which is perpendicular to the $x_n$ axis to be the $x_1$ direction, without loss of generality, we may suppose that
$$
T_\lambda=\{x=(x_1,x')\in\mathbb{R}^n_+|x_1=\lambda ~\text{for~some}~ \lambda\in\mathbb{R}\},
$$
where $x'=(x_2,x_3,\cdots,x_n)$. Let
$$
x^\lambda=(2\lambda-x_1, x_2, \cdots, x_n)
$$
be the reflection of $x$ about the plane $T_\lambda$. Denote
$$
\Sigma_\lambda=\{x\in\mathbb{R}^n_+|x_1<\lambda\}~~\text{and}~~\tilde{\Sigma}_\lambda=\{x|x^\lambda\in \Sigma_\lambda\}.
$$

In \cite{book}, Chen, Li and Ma introduced two key ingredients for the direct method of moving planes for the fractional Laplacians defined in (\ref{pp}). Based on their results, here we also establish two similar key ingredients, in which the fractional Laplacians are differently defined as (\ref{_+_}). Since the proofs of the two key ingredients here are almost the same as the ones in \cite{book}, we only prove them briefly for readers convenience, for the detailed proofs, please refer to \cite{book}.

\begin{thm}(Decay at Infinity) Let $\Omega$ be an unbounded region in $\Sigma_\lambda$. Assume that $U\in C^{1,1}_{loc}(\Omega)\cap\cal{L}_\alpha$$(\mathbb{R}^n_+)$, $V\in C^{1,1}_{loc}(\Omega)\cap\cal{L}_\beta$$(\mathbb{R}^n_+)$ are lower semi-continuous on $\bar{\Omega}$ for $0<\alpha,\beta<2$, and satisfy
\begin{equation}
\left\{\begin{array}{llll}
(-\lap)^{\alpha/2}U(x)+c_1(x)V(x)\geq0, & \\
(-\lap)^{\beta/2}V(x)+c_2(x)U(x)\geq0, & \qquad in ~\Omega, \\
U(x),V(x)\geq0, & \qquad in ~\Sigma_\lambda\backslash\Omega,\\
U(x^\lambda)=-U(x), & \\
V(x^\lambda)=-V(x), & \qquad in ~\Sigma_\lambda,
\end{array}\right.
\label{thm2.131}
\end{equation}
with $c_i(x)\leq0,~i=1,2,$ and
\begin{equation}
\displaystyle\underset{|x|\rightarrow\infty}{\underline{\lim}}|x|^\alpha c_1(x)=0,~~~\displaystyle\underset{|x|\rightarrow\infty}{\underline{\lim}}|x|^\beta c_2(x)=0,
\label{thm2.16}
\end{equation}
then there exists a constant $R_0>0$ (depending on $c_i(x)$, but is independent of $U,V$) such that if
\begin{equation}
U(\tilde{x})=\min_{\bar{\Omega} }U(x)<0,~~~V(\bar{x})=\min_{\bar{\Omega}} V(x)<0
\label{thm2.17}
\end{equation}
then
\begin{equation}
|\tilde{x}|,|\bar{x}|\leq R_0.
\label{ayy00}
\end{equation}
\label{thm2.3}
\end{thm}

\begin{proof} If there is a point $\tilde{x}\in \Sigma_\lambda$ such that
\begin{equation}
U(\tilde{x})=\min_{\bar{\Omega}} U(x)<0,
\end{equation}
then we have by (\ref{_+_}) that
\begin{eqnarray}
(-\lap)^{\alpha/2}U(\tilde{x})&=&C_{n,\alpha}PV\int_{\mathbb{R}^n_+}
\frac{U(\tilde{x})-U(y)}{|\tilde{x}-y|^{n+\alpha}}dy
\nonumber\\
&=&C_{n,\alpha}PV\left\{\int_{\Sigma_\lambda}
\frac{U(\tilde{x})-U(y)}{|\tilde{x}-y|^{n+\alpha}}dy+\int_{\mathbb{R}^n_+\setminus \Sigma_\lambda}
\frac{U(\tilde{x})-U(y)}{|\tilde{x}-y|^{n+\alpha}}dy\right\}
\nonumber\\
&=&C_{n,\alpha}PV\left\{\int_{\Sigma_\lambda}
\frac{U(\tilde{x})-U(y)}{|\tilde{x}-y|^{n+\alpha}}dy+\int_{\Sigma_\lambda}
\frac{U(\tilde{x})+U(y)}{|\tilde{x}-y^\lambda|^{n+\alpha}}dy\right\}
\label{ayy1}\\
&\leq&C_{n,\alpha}\int_{\Sigma_\lambda}\left\{
\frac{U(\tilde{x})-U(y)}{|\tilde{x}-y^\lambda|^{n+\alpha}}+
\frac{U(\tilde{x})+U(y)}{|\tilde{x}-y^\lambda|^{n+\alpha}}\right\}dy
\nonumber\\
&=&C_{n,\alpha}\int_{\Sigma_\lambda}
\frac{2U(\tilde{x})}{|\tilde{x}-y^\lambda|^{n+\alpha}}dy.
\label{ayy2}
\end{eqnarray}
Denote $D_1=(B_{2|\tilde{x}|}(\tilde{x})\setminus B_{|\tilde{x}|}(\tilde{x}))\cap \tilde{\Sigma}_\lambda$. For each fixed $\lambda<0$, there exists a $C>0$ such that
\begin{eqnarray}
\int_{\Sigma_\lambda}\frac{1}{|\tilde{x}-y^\lambda|^{n+\alpha}}dy
=\int_{\tilde{\Sigma}_\lambda}\frac{1}{|\tilde{x}-y|^{n+\alpha}}dy
\geq\int_{D_1}\frac{1}{|\tilde{x}-y|^{n+\alpha}}dy
\sim\frac{C}{|\tilde{x}|^\alpha}.
\label{ayy3}
\end{eqnarray}
Hence
\begin{equation}
(-\lap)^{\alpha/2}U(\tilde{x})\leq\frac{CU(\tilde{x})}{|\tilde{x}|^\alpha}<0.
\label{ayy4}
\end{equation}
Together with $(-\lap)^{\alpha/2}U(x)+c_1(x)V(x)\geq0$ in (\ref{0}), and $c_1(x)\leq0$, it derives
\begin{equation}
V(\tilde{x})<0,
\label{ayy5}
\end{equation}
and
\begin{equation}
U(\tilde{x})\geq-Cc_1(\tilde{x})|\tilde{x}|^\alpha V(\tilde{x}).
\label{ayy6}
\end{equation}
Since $V$ is lower semi-continuous on $\bar{\Omega}$, it sees from (\ref{ayy5}) that there is a $\bar{x}\in \Sigma_\lambda$ such that
\begin{equation}
V(\bar{x})=\underset{\bar{\Omega}}{\min}V(x)<0,
\label{ayy7}
\end{equation}
and similarly there also holds
\begin{equation}
(-\lap)^{\beta/2}V(\bar{x})\leq\frac{CV(\bar{x})}{|\bar{x}|^\beta}<0.
\label{ayy8}
\end{equation}
Therefore, for fixed $\lambda<0$, if $|\tilde{x}|,~|\bar{x}|$ are sufficiently large, then by (\ref{thm2.131}), (\ref{thm2.16}), (\ref{ayy6}) and (\ref{ayy8}), one can deduce that
\begin{eqnarray}
0&\leq&(-\lap)^{\beta/2}V(\bar{x})+c_2(\bar{x})U(\bar{x})\nonumber\\
&\leq&\frac{CV(\bar{x})}{|\bar{x}|^\beta}+c_2(\bar{x})U(\tilde{x})\nonumber\\
&\leq&C\left(\frac{V(\bar{x})}{|\bar{x}|^\beta}-c_2(\bar{x})
c_1(\tilde{x})|\tilde{x}|^\alpha V(\tilde{x})\right)\nonumber\\
&\leq&C\left(\frac{V(\bar{x})}{|\bar{x}|^\beta}-c_2(\bar{x})
c_1(\tilde{x})|\tilde{x}|^\alpha V(\bar{x})\right)\nonumber\\
&\leq&\frac{CV(\bar{x})}{|\bar{x}|^\beta}\left(1-c_2(\bar{x})
c_1(\tilde{x})|\tilde{x}|^\alpha|\bar{x}|^\beta\right)\nonumber\\
&<&0.
\label{ayy9}
\end{eqnarray}
This contradiction shows that $U(x)\geq0$ for $|x|$ sufficiently large. Similarly, one can also deduce that $V(x)\geq0$ for $|x|$ large enough. This indicates that $U(x),V(x)\geq0$ when $|x|>R$, $R>0$ large, and finally verifies (\ref{ayy00}).

\end{proof}

\begin{thm}(Narrow Region Principle) Let $\Omega\subset\Sigma_\lambda$ be a bounded narrow region contained in $\{x|\lambda-l<x_1<\lambda\}$ with small $l>0$. Suppose that $c_i(x)\leq0,~i=1,2,$ are bounded from below in $\Omega$, $U\in C^{1,1}_{loc}(\Omega)\cap\cal{L}_\alpha$$(\mathbb{R}^n)_+$ and $V\in C^{1,1}_{loc}(\Omega)\cap\cal{L}_\beta$$(\mathbb{R}^n_+)$ are lower semi-continuous on $\bar{\Omega}$ for $0<\alpha,\beta<2$, and satisfy
\begin{equation}
\left\{\begin{array}{llll}
(-\lap)^{\alpha/2}U(x)+c_1(x)V(x)\geq0, & \\
(-\lap)^{\beta/2}V(x)+c_2(x)U(x)\geq0, & \qquad in ~\Omega, \\
U(x),V(x)\geq0, & \qquad in ~\Sigma_\lambda\backslash\Omega,\\
U(x^\lambda)=-U(x), & \\
V(x^\lambda)=-V(x), & \qquad in ~\Sigma_\lambda,
\end{array}\right.
\label{thm2.13}
\end{equation}
then we have for sufficiently small $l>0$ that
\begin{description}
\item[~~~\emph{(i)}]
\begin{equation}
U(x),V(x)\geq0~~in~\Omega;
\label{thm2.14}
\end{equation}
\item[~~~\emph{(ii)}] Furthermore, if $U$ or $V$ attains $0$ at some point in $\Sigma_\lambda$, then
\begin{equation}
U(x)=V(x)\equiv0~~\text{almost everywhere in}~~\mathbb{R}^n_+;
\label{thm2.15555555555}
\end{equation}
\item[~~~\emph{(iii)}] Conclusions (i) and (ii) hold for the unbounded region $\Omega$ if we further assume that
$$
\displaystyle\underset{|x| \ra \infty}{\underline{\lim}}U(x),V(x)\geq0.
$$
\end{description}
\label{thm2.2}
\end{thm}

\begin{proof}
Let us prove by the contrary. If (\ref{thm2.14}) does not hold, then by the semi-continuity of $U$ on $\bar{\Omega}$, there is a $\tilde{x}\in\bar{\Omega}$ such that
\begin{equation}
U(\tilde{x})=\min_{\bar{\Omega}}U<0,
\end{equation}
and by (\ref{thm2.13}), one can easily deduce that $\tilde{x}$ is in the interior of $\Omega$. Similar to (\ref{ayy2}), we have
\begin{equation}
(-\lap)^{\alpha/2}U(\tilde{x})\leq C_{n,\alpha}\int_{\Sigma_\lambda}
\frac{2U(\tilde{x})}{|\tilde{x}-y^\lambda|^{n+\alpha}}dy.
\end{equation}
and let $D_2=(B_{2l}(\tilde{x})\setminus B_{l}(\tilde{x}))\cap \tilde{\Sigma}_\lambda$, it derives
\begin{eqnarray}
\int_{\Sigma_\lambda}\frac{1}{|\tilde{x}-y^\lambda|^{n+\alpha}}dy
\geq\frac{C}{l^\alpha}.
\end{eqnarray}
Thus, similar to the proof of Theorem \ref{thm2.3}, it holds
\begin{equation}
(-\lap)^{\alpha/2}U(\tilde{x})\leq\frac{CU(\tilde{x})}{l^\alpha}<0,
\label{ayy21}
\end{equation}
and
\begin{equation}
U(\tilde{x})\geq-Cc_1(\tilde{x})l^\alpha V(\tilde{x}),
\label{ayy22}
\end{equation}
which indicates $V(\tilde{x})<0$, then there exists a $\bar{x}\in\Omega$ such that
\begin{equation}
V(\bar{x})=\min_{\bar{\Omega}}V(x)<0,
\end{equation}
and similar to (\ref{ayy21}), we can also get
\begin{equation}
(-\lap)^{\beta/2}V(\bar{x})\leq\frac{CV(\bar{x})}{l^\beta}<0.
\end{equation}
Together with (\ref{ayy22}), similar to (\ref{ayy9}), since $c_i(x)\leq0,~i=1,2,$ are bounded from below, then we have for $l>0$ sufficiently small that
\begin{eqnarray}
0&\leq&(-\lap)^{\beta/2}V(\bar{x})+c_2(\bar{x})U(\bar{x})\nonumber\\
&\leq&\frac{CV(\bar{x})}{|\bar{x}|^\beta}\left(1-c_2(\bar{x})
c_1(\tilde{x})|\tilde{x}|^\alpha|\bar{x}|^\beta\right)
<0.
\end{eqnarray}
This contradiction shows that $U(x)\geq0$ in $\Omega$. Similarly one can also obtain that $V(x)\geq0$ in $\Omega$, Thus (\ref{thm2.14}) must be true.

To prove Theorem \ref{thm2.2} $(ii)$, without loss of generality, we may suppose that $U(x^0)=0$ for some $x^0\in\Omega$, then similar to (\ref{ayy1}),
\begin{eqnarray}
0&\leq&(-\lap)^{\alpha/2}U(x^0)\nonumber\\
&=&C_{n,\alpha}PV\left\{\int_{\Sigma_\lambda}
\frac{U(\tilde{x})-U(y)}{|\tilde{x}-y|^{n+\alpha}}dy+\int_{\Sigma_\lambda}
\frac{U(\tilde{x})+U(y)}{|\tilde{x}-y^\lambda|^{n+\alpha}}dy\right\}
\nonumber\\
&=&C_{n,\alpha}PV\int_{\Sigma_\lambda}
\left\{\frac{1}{|\tilde{x}-y^\lambda|^{n+\alpha}}-
\frac{1}{|\tilde{x}-y|^{n+\alpha}}\right\}U(y)dy\nonumber\\
&\leq&0,
\end{eqnarray}
which implies $U\equiv0$ in $\Sigma_\lambda$, since $U(x^\lambda)=-U(x)$ in $\Sigma_\lambda$, then $$U(x)\equiv0~~~\text{in}~~ \mathbb{R}^n_+.$$

Again since $(-\lap)^{\alpha/2}U(x)+c_1(x)V(x)=c_1(x)V(x)\geq0$, $c_1(x)\leq0$, then $V(x)\leq0$ in $\Sigma_\lambda,$ and because we already know $V(x)\geq0$ in $\Sigma_\lambda,$ hence $V(x)\equiv0,~x\in \Sigma_\lambda$. Now from the anti-symmetry of $V(x)$, we obtain that $$V(x)\equiv0~~~\text{in}~~ \mathbb{R}^n_+.$$

Similarly, one can also prove that if $V(x)=0$ somewhere in $\Sigma_\lambda$, then $U=V\equiv0$ in $\mathbb{R}^n_+$.

The proof of Theorem \ref{thm2.2} completes here.
\end{proof}

\begin{remark} Suppose that the minimum point of $U$ in $\bar{\Omega}$ is $\tilde{x}$, and the minimum point of $V$ in $\bar{\Omega}$ is $\bar{x}$. From the proofs of Theorem \ref{thm2.3} and Theorem \ref{thm2.2}, we can see that the inequalities
\begin{equation}
c_1(x)\leq0,~~~\displaystyle\underset{|x|\rightarrow\infty}{\underline{\lim}}|x|^\alpha c_1(x)=0,~~~(-\lap)^{\alpha/2}U(x)+c_1(x)V(x)\geq0,
\end{equation}
and
\begin{equation}
c_2(x)\leq0,~~~\displaystyle\underset{|x|\rightarrow\infty}{\underline{\lim}}|x|^\beta c_2(x)=0,~~~(-\lap)^{\beta/2}V(x)+c_2(x)U(x)\geq0
\end{equation}
are only required at $\tilde{x}$ and $\bar{x}$ respectively. Furthermore, if $U(\tilde{x})<0$, then $V(\tilde{x})<0$; and also, if $V(\bar{x})<0$, then $V(\bar{x})<0$.
\label{rem2.1}
\end{remark}

\section{The proof of Theorem \ref{thm1}}

In this section, we prove Theorem \ref{thm1}.

\begin{proof} Suppose $u,v$ are the solutions to (\ref{0}), we first show that either
\begin{equation}
u(x)=v(x)\equiv0~\text{in}~\mathbb{R}^n_+,~\text{and}~f(0)=g(0)=0,
\label{equal0}
\end{equation}
or
\begin{equation}
u(x),v(x)>0~ ~~\text{in} ~\mathbb{R}^n_+.
\label{larger0}
\end{equation}
Since $u(x),v(x)\geq0$ in $\mathbb{R}^n_+$, then from (\ref{0}) and Theorem \ref{thm2.1} \textbf{(ii)}, we know that
$$
\text{either}~ u\equiv0~\text{ or}~ u>0 ~\text{in}~ \mathbb{R}^n_+.
$$
If $u\equiv0$ in $\mathbb{R}^n_+$, by (\ref{0}) and (\ref{_+_}) one has $f(v(x))\equiv0$ in $\mathbb{R}^n_+$, together with $f(0)\geq0$ and $f(t)$ is strictly increasing in $t\geq0$, one can obtain that $f(0)=0$, $v\equiv0$ in $\mathbb{R}^n_+$, and thus $g(0)=0$. Hence if $u\equiv0$ in $\mathbb{R}^n_+$, then $v\equiv0$ in $\mathbb{R}^n_+$, and $f(0)=g(0)=0$. Similarly, one can also derive that if $v\equiv0$ in $\mathbb{R}^n_+$, then $u\equiv0$ in $\mathbb{R}^n_+$, and $f(0)=g(0)=0$.

Till now, we have shown that either (\ref{equal0}) or (\ref{larger0}) holds. Therefore, without loss of generality, we may now suppose $u,v>0$ in $\mathbb{R}^n_+$, and go on to prove $u=u(x_n), v=v(x_n)$. Finally we verify that this is a contradiction with the finiteness of $u,v$, and thus arrive at the desired conclusion.

Because there is no decay condition on $u,v$ near infinity, we are not able to carry the method of moving planes on $u,v$ directly. To overcome this difficulty, we make a Kelvin transform. And to guarantee that $\mathbb{R}^n_+$ is invariant under the transform, we need to locate the center $x^0$ on the boundary $\partial\mathbb{R}^n_+=\{x|x_n=0\}$.

Now for any point $x^0\in\partial\mathbb{R}^n_+$, let $\bar{u}$ and $\bar{v}$
\begin{equation}
\left\{\begin{array}{ll}
\bar{u}(x)=\frac{1}{|x-x^0|^{n-\alpha}}u(x^0+\frac{x-x^0}{|x-x^0|^2}),
~~x\in\mathbb{R}^n_+,\\
\bar{v}(x)=\frac{1}{|x-x^0|^{n-\beta}}v(x^0+\frac{x-x^0}{|x-x^0|^2}),
~~x\in\mathbb{R}^n_+,
\end{array}\right.
\label{e3.1}
\end{equation}
be the Kelvin transforms of $u$ and $v$ centered at $x^0$ respectively. Then by (\ref{0}) and (\ref{e3.1}), one has
\begin{equation}
\left\{\begin{array}{ll}
(-\lap)^{\alpha/2} \bar{u}(x)=\frac{1}{|x-x^0|^{n+\alpha}}f(|x-x^0|^{n-\beta}\bar{v}(x)),\quad x\in\mathbb{R}^n_+,\\
(-\lap)^{\beta/2} \bar{v}(x)=\frac{1}{|x-x^0|^{n+\beta}}g(|x-x^0|^{n-\alpha}\bar{u}(x)),\quad x\in\mathbb{R}^n_+.
\end{array}\right.
\label{3002}
\end{equation}
The functions $\bar{u},\bar{v}\in C_{loc}^{1,1}(\mathbb{R}^n_+)$ are positive, decay to $0$ at infinity no slower than $|x-x^0|^{\alpha-n}$ and $|x-x^0|^{\beta-n}$ respectively, and may have a singularity at $x^0$. Since $u,v\in C(\mathbb{R}^n)$, together with (\ref{e3.1}), we know that $\bar{u},\bar{v}\rightarrow0$ when $x_n\rightarrow0$ and $|x-x^0|\geq\epsilon>0$, where $\epsilon$ is a small positive constant. In the following, we prove that $\bar{u},~\bar{v}$ are rotationally symmetric about any line parallel to the $x_n$-axis, from which we can derive that $u=u(x_n),v=v(x_n)$ in $\mathbb{R}^n_+$.

For $x=(x_1,x_2,\cdots,x_n)$, choose any direction which is perpendicular to the $x_n$-axis to be the $x_1$ direction. To prove that $\bar{u},\bar{v}$ are rotationally symmetric about some line parallel to the $x_n$-axis, it suffices to show that $\bar{u},\bar{v}$ are symmetric in $x_1$.

Write $x^0=(x^0_1,x^0_2,\cdots,x^0_n)$, for any $\lambda<x^0_1$, $\lambda\in\mathbb{R}$, we already defined $T_\lambda,x^\lambda,\Sigma_\lambda$ and $\tilde{\Sigma}_\lambda$ in Section 2. Now let
$$
\bar{u}_\lambda(x):=\bar{u}(x^\lambda),~~\bar{v}_\lambda(x):=\bar{v}(x^\lambda),
$$
$$
U_\lambda(x)=\bar{u}_\lambda(x)-\bar{u}(x),
~~V_\lambda(x)=\bar{v}_\lambda(x)-\bar{v}(x),
$$
and
$$
\Sigma_{U_\lambda}^-=\{x\in\Sigma_\lambda|U_\lambda(x)<0\},~~
\Sigma_{V_\lambda}^-=\{x\in\Sigma_\lambda|V_\lambda(x)<0\}.
$$
Then the functions $U_\lambda(x)$, $V_\lambda(x)$ satisfy
\begin{equation}
\left\{\begin{array}{ll}
(-\lap)^{\alpha/2} U_\lambda(x)=\frac{f(|x^\lambda-x^0|^{n-\beta}\bar{v}_\lambda(x))}{|x^\lambda-x^0|^{n+\alpha}}
-\frac{f(|x-x^0|^{n-\beta}\bar{v}(x))}{|x-x^0|^{n+\alpha}},\quad x\in\Sigma_\lambda,\\
(-\lap)^{\beta/2} V_\lambda(x)=\frac{g(|x^\lambda-x^0|^{n-\alpha}\bar{u}_\lambda(x))}{|x^\lambda-x^0|^{n+\beta}}
-\frac{g(|x-x^0|^{n-\alpha}\bar{u}(x))}{|x-x^0|^{n+\beta}},\quad x\in\Sigma_\lambda.
\end{array}\right.
\label{32}
\end{equation}

\textbf{Step 1.} \textsl{Start moving the plane $T_\lambda$ from $-\infty$ to the right along the $x_1$ direction.}

In the following, with the help of Theorem \ref{thm2.3}, we prove that for $\lambda<x^0_1$, $|\lambda|$ large enough,
\begin{equation}
U_\lambda,V_\lambda\geq0~~ \text{in}~~ \Sigma_\lambda.
\label{y1}
\end{equation}

It is easy to see that $U_\lambda(x)=V_\lambda(x)=0$ on $T_\lambda$, and
$$
\lim_{|x|\rightarrow\infty}U_\lambda(x),V_\lambda(x)=0,~~
\lim_{x_n\rightarrow0}U_\lambda(x),V_\lambda(x)=0~\text{for}~
x\in\mathbb{R}^n_+\setminus B_\epsilon((x^0)^\lambda).
$$
Let us first admit the following claim, its proof will be given in Section 4.

\textbf{Claim 3.1.} \textsl{For $\lambda<x_1^0$, $|\lambda|$ large enough, there exists a small constant $\epsilon>0$ and a positive constant $c_\lambda$ such that}
\begin{equation}
U_\lambda(x),V_\lambda(x)\geq c_\lambda,~~~~x\in B_\epsilon((x^0)^\lambda)\cap\mathbb{R}^n_+.
\label{0360}
\end{equation}
Now we know that if $U_\lambda$ and $V_\lambda$ are negative somewhere in $\Sigma_\lambda$, then the negative minima of $U_\lambda$, $V_\lambda$ are attained in the interior of $\Sigma_\lambda\setminus B_\epsilon((x^0)^\lambda)$. Hence if $\Sigma_{U_\lambda}^-\neq\emptyset,$ then there is a point $\tilde{x}$ such that
$$
U_\lambda(\tilde{x})=\min_{\Sigma_\lambda}U_\lambda<0,
$$
and similar to (\ref{ayy4}), we also obtain
\begin{equation}
(-\lap)^{\alpha/2}U_\lambda(\tilde{x})\leq\frac{CU_\lambda(\tilde{x})}{|\tilde{x}|^\alpha}<0.
\end{equation}
Because
\begin{eqnarray}
&&(-\lap)^{\alpha/2}U_\lambda(\tilde{x})\nonumber\\ &=&\frac{f(|\tilde{x}^\lambda-x^0|^{n-\beta}\bar{v}_\lambda(\tilde{x}))}{|\tilde{x}^\lambda-x^0|^{n+\alpha}}-
\frac{f(|\tilde{x}-x^0|^{n-\beta}\bar{v}(\tilde{x}))}{|\tilde{x}-x^0|^{n+\alpha}}\nonumber\\
&=&\left(\frac{f(|\tilde{x}^\lambda-x^0|^{n-\beta}\bar{v}_\lambda(\tilde{x}))}
{[|\tilde{x}^\lambda-x^0|^{n-\beta}\bar{v}_\lambda(\tilde{x})]^{p_0}}-
\frac{f(|\tilde{x}-x^0|^{n-\beta}\bar{v}_\lambda(\tilde{x}))}
{[|\tilde{x}-x^0|^{n-\beta}\bar{v}_\lambda(\tilde{x})]^{p_0}}\right)\bar{v}^{p_0}_\lambda(\tilde{x})
\nonumber\\&&+\frac{f(|\tilde{x}-x^0|^{n-\beta}\bar{v}_\lambda(\tilde{x}))-f(|\tilde{x}-x^0|^{n-\beta}\bar{v}(\tilde{x}))}{|\tilde{x}-x^0|^{n+\alpha}}
\nonumber\\
&:=&I_1+I_2<0,
\label{33}
\end{eqnarray}
and obviously $I_1\geq0$, so $I_2<0$. Now by the strict monotonicity of $f$, one can arrive at $V_\lambda(\tilde{x})<0$. Then there is a point $\bar{x}\in\Sigma_\lambda$ such that
$$
V_\lambda(\bar{x})=\min_{\Sigma_\lambda}V_\lambda(x)<0,
$$
and furthermore, we can also deduce that $U_\lambda(\bar{x})<0$.

Similarly, one can also derive that, if $\Sigma_{V_\lambda}^-\neq\emptyset,$ then there exists points $\bar{x}$ and $\tilde{x}$ such that
\begin{equation}
\left\{\begin{array}{ll}
V_\lambda(\bar{x})=\underset{\Sigma_\lambda}{\min}V_\lambda(x)<0
~~\text{and}~~U(\bar{x})<0;\\
U_\lambda(\tilde{x})=\underset{\Sigma_\lambda}{\min}U_\lambda(x)<0
~~\text{and}~~V(\tilde{x})<0.
\end{array}\right.
\label{3201}
\end{equation}

Since $f(t)\geq0$ are strictly increasing about $t$ in $[0,+\infty)$, $\frac{f(t)}{t^{p_0}}$ is non-increasing in $t>0$ with $p_0=\frac{n+\alpha}{n-\beta}$, then it follows from (\ref{32}), (\ref{3201}) that
\begin{eqnarray}
(-\lap)^{\alpha/2}U_\lambda(\tilde{x})
&=&\frac{f(|\tilde{x}^\lambda-x^0|^{n-\beta}\bar{v}_\lambda(\tilde{x}))}{|\tilde{x}^\lambda-x^0|^{n+\alpha}}-
\frac{f(|\tilde{x}-x^0|^{n-\beta}\bar{v}(\tilde{x}))}{|\tilde{x}-x^0|^{n+\alpha}}\nonumber\\
&=&(\frac{f(|\tilde{x}^\lambda-x^0|^{n-\beta}\bar{v}_\lambda(\tilde{x}))}
{[|\tilde{x}^\lambda-x^0|^{n-\beta}\bar{v}_\lambda(\tilde{x})]^{p_0}}\bar{v}^{p_0}_\lambda(\tilde{x})-
\frac{f(|\tilde{x}-x^0|^{n-\beta}\bar{v}(\tilde{x}))}
{[|\tilde{x}-x^0|^{n-\beta}\bar{v}(\tilde{x})]^{p_0}}\bar{v}^{p_0}(\tilde{x})
\nonumber\\
&\geq&\frac{f(|\tilde{x}-x^0|^{n-\beta}\bar{v}(\tilde{x}))}
{[|\tilde{x}-x^0|^{n-\beta}\bar{v}(\tilde{x})]^{p_0}}
[\bar{v}_\lambda^{p_0}(\tilde{x})-v^{p_0}(\tilde{x})]\nonumber\\
&=&\frac{f(|\tilde{x}-x^0|^{n-\beta}\bar{v}(\tilde{x}))}
{[|\tilde{x}-x^0|^{n-\beta}\bar{v}(\tilde{x})]^{p_0}}
p_0\eta^{p_0-1}(\tilde{x})V_\lambda(\tilde{x}),~~~
v_\lambda(\tilde{x})<\eta(\tilde{x})<v(\tilde{x})
\nonumber\\
&\geq&\frac{f(|\tilde{x}-x^0|^{n-\beta}\bar{v}(\tilde{x}))}
{[|\tilde{x}-x^0|^{n-\beta}\bar{v}(\tilde{x})]^{p_0}}
p_0\bar{v}^{p_0-1}(\tilde{x})V_\lambda(\tilde{x}).
\label{3301}
\end{eqnarray}
Denote $s=|\tilde{x}-x^0|^{n-\beta}\bar{v}(\tilde{x})>0$, because we already know that $\tilde{x}$ is in the interior of $\Sigma_\lambda\setminus B_\epsilon((x^0)^\lambda)$, so $s\geq c$ for some positive constant $c$. Since $\frac{f(t)}{t^{p_0}}$ is non-increasing in $t>0$, then there is a constant $C>0$ such that
$$
\frac{f(s)}{s^{p_0}}\leq\frac{f(c)}{c^{p_0}}\leq C.
$$
Together with (\ref{3301}), we have
\begin{equation}
(-\lap)^{\alpha/2}U_\lambda(\tilde{x})
\geq C\bar{v}^{p_0-1}(\tilde{x})V_\lambda(\tilde{x}).
\label{33001}
\end{equation}
Similarly, we can also obtain
\begin{equation}
(-\lap)^{\alpha/2}V_\lambda(\bar{x})
\geq C\bar{u}^{q_0-1}(\bar{x})U_\lambda(\bar{x}).
\label{3302}
\end{equation}

Let
\begin{equation}
c_1(\tilde{x})=-C\bar{v}^{p_0-1}(\tilde{x}),~~
c_2(\bar{x})=-C\bar{u}^{q_0-1}(\bar{x}).
\label{3303}
\end{equation}
Then at the minimum point of $\bar{u}$ ($x$) and the minimum point of $\bar{v}$ ($x$), we have
\begin{equation}
\left\{\begin{array}{ll}
(-\lap)^{\alpha/2} U_\lambda(\tilde{x})+c_1(\tilde{x})V_\lambda(\tilde{x})\geq0,\\
(-\lap)^{\beta/2} V_\lambda(\bar{x})+c_2(\bar{x})U_\lambda(\bar{x})\geq0,
\end{array}\right.
\label{3203}
\end{equation}
with $c_i(x)<0$, $i=1,2.$

From (\ref{e3.1}) and (\ref{3303}), it is easy to verify that,
\begin{equation}
\displaystyle\underset{|\tilde{x}|\rightarrow\infty}{\underline{\lim}}|\tilde{x}|^\alpha c_1(\tilde{x})=0,~~~\displaystyle\underset{|\bar{x}|\rightarrow\infty}{\underline{\lim}}|\bar{x}|^\beta c_2(\bar{x})=0,
\label{36}
\end{equation}
hence $c_1(\tilde{x}),c_2(\bar{x})<0$ satisfy (\ref{thm2.16}). Now by (\ref{0360}) and Remark \ref{rem2.1}, applying Theorem \ref{thm2.3} to $U_\lambda$, $V_\lambda$ with $\Omega=(\Sigma_{U_\lambda}^-\cup\Sigma_{V_\lambda}^-)$, we conclude that there exists $R_0>0$ (independent of $\lambda$), such that if $\tilde{x}$, $\bar{x}$ are negative minima of $U_\lambda$, $V_\lambda$ respectively in $\Sigma_\lambda$, then
\begin{equation}
|\tilde{x}|,|\bar{x}|\leq R_0.
\label{37}
\end{equation}
Now for $\lambda\leq-R_0$, we have (\ref{y1}).

\textbf{Step 2.} \textsl{Keep moving the plane $T_\lambda$ to the right until arriving at the limiting position.}

\emph{Step 1} provides a starting point from which we can move the plane $T_\lambda$ to the right as long as (\ref{y1}) holds to its limiting position. Denote
$$
\lambda_0=\sup\{\lambda<x^0_1|U_\mu(x),V_\mu(x)\geq0, ~x\in\Sigma_\mu,~\mu\leq\lambda\}.
$$
We show
\begin{equation}
U_{\lambda_0}(x)=V_{\lambda_0}(x)\equiv0, ~~~x\in \Sigma_{\lambda_0}.
\label{39}
\end{equation}
To prove (\ref{39}), let us consider two possibilities:
\begin{description}
\item[~~~(i)]$~\lambda_0<x^0_1$;
\item[~~~(ii)]$~\lambda_0=x^0_1$.
\end{description}

\textbf{Possibility (\textbf{i}).} For $\lambda_0<x^0_1,$ we show that if (\ref{39}) does not hold, then the plane $T_\lambda$ can be moved further to the right. To be more rigorous, we prove that there exists some $\varepsilon>0$ such that for any $\lambda\in(\lambda_0,\lambda_0+\varepsilon)$,
\begin{equation}
U_{\lambda}(x),V_{\lambda}(x)\geq0, ~~~x\in \Sigma_{\lambda}.
\label{310}
\end{equation}
This is a contradiction with the definition of $\lambda_0$. Hence we must have (\ref{39}).

Now we prove (\ref{310}) by combining Theorem \ref{thm2.3} and Theorem \ref{thm2.2} under the assumption that (\ref{39}) is not true. In fact, since $\lambda_0<x_1^0$, we have
\begin{equation}
U_{\lambda_0}(x),V_{\lambda_0}(x)>0, ~~~x\in \Sigma_{\lambda_0}.
\label{3100}
\end{equation}
If not, without loss of generality, we may suppose that there is a point $\tilde{x}$ such that
$$
U_{\lambda_0}(\tilde{x})=\min_{\Sigma_{\lambda_0}}U_{\lambda_0}=0.
$$
By the definition of $\lambda_0$, one knows that $U_{\lambda_0}(x),V_{\lambda_0}(x)\geq0$ in $\Sigma_{\lambda_0}$, then on the one hand,
\begin{eqnarray}
(-\lap)^{\alpha/2}U_{\lambda_0}(\tilde{x})&=&C_{n,\alpha}PV\int_{\mathbb{R}^n_+}
\frac{U_{\lambda_0}(\tilde{x})-U_{\lambda_0}(y)}{|\tilde{x}-y|^{n+\alpha}}dy\nonumber\\
&=&C_{n,\alpha}PV\int_{\Sigma_{\lambda_0}}
\frac{-U_{\lambda_0}(y)}{|\tilde{x}-y|^{n+\alpha}}dy+C_{n,\alpha}PV\int_{\tilde{\Sigma}_{\lambda_0}}
\frac{-U_{\lambda_0}(y)}{|\tilde{x}-y|^{n+\alpha}}dy\nonumber\\
&=&C_{n,\alpha}PV\int_{\Sigma_{\lambda_0}}
\frac{-U_{\lambda_0}(y)}{|\tilde{x}-y|^{n+\alpha}}dy+C_{n,\alpha}PV\int_{\Sigma_{\lambda_0}}
\frac{U_{\lambda_0}(y)}{|\tilde{x}-y^{\lambda_0}|^{n+\alpha}}dy\nonumber\\
&=&C_{n,\alpha}PV\int_{\Sigma_{\lambda_0}}
\left(\frac{1}{|\tilde{x}-y^{\lambda_0}|^{n+\alpha}}-\frac{1}{|\tilde{x}-y|^{n+\alpha}}\right)
U_{\lambda_0}(y)dy\nonumber\\
&<&0,
\label{3.20}
\end{eqnarray}
on the other hand,
\begin{eqnarray}
(-\lap)^{\alpha/2}U_{\lambda_0}(\tilde{x})
&=&\frac{f(|\tilde{x}^{\lambda_0}-x^0|^{n-\beta}\bar{v}_{\lambda_0}(\tilde{x}))}{|\tilde{x}^{\lambda_0}-x^0|^{n+\alpha}}-
\frac{f(|\tilde{x}-x^0|^{n-\beta}\bar{v}(\tilde{x}))}{|\tilde{x}-x^0|^{n+\alpha}}\nonumber\\
&=&\frac{f(|\tilde{x}^{\lambda_0}-x^0|^{n-\beta}\bar{v}_{\lambda_0}(\tilde{x}))}
{[|\tilde{x}^{\lambda_0}-x^0|^{n-\beta}\bar{v}_{\lambda_0}(\tilde{x})]^{p_0}}\bar{v}_{\lambda_0}^{p_0}(\tilde{x})-
\frac{f(|\tilde{x}-x^0|^{n-\beta}\bar{v}_{\lambda_0}(\tilde{x}))}
{[|\tilde{x}-x^0|^{n-\beta}\bar{v}_{\lambda_0}(\tilde{x})]^{p_0}}\bar{v}_{\lambda_0}^{p_0}(\tilde{x})\nonumber\\
&&+\frac{f(|\tilde{x}-x^0|^{n-\beta}\bar{v}_{\lambda_0}(\tilde{x}))}
{|\tilde{x}-x^0|^{n+\alpha}}-
\frac{f(|\tilde{x}-x^0|^{n-\beta}\bar{v}(\tilde{x}))}
{|\tilde{x}-x^0|^{n+\alpha}}\nonumber\\
&\geq&0.
\label{3.21}
\end{eqnarray}
Here we arrive at a contradiction, and this verifies (\ref{3100}).

Let us first admit

\textbf{Claim 3.2.} \textsl{There exists a constant $c_0>0$ such that for sufficiently small $\eta>0$,
\begin{equation}
U_{\lambda_0}(x),~V_{\lambda_0}(x)\geq c_0,~~x\in B_\eta((x^0)^{\lambda_0})\cap\mathbb{R}^n_+.
\label{y2}
\end{equation}}
We will prove Claim 3.2 in Section 4. Now for any $x\in B_\eta((x^0)^{\lambda})\cap\mathbb{R}^n_+$, let us find $\hat{x}\in B_\eta((x^0)^{\lambda_0})\cap\mathbb{R}^n_+$ such that $x^\lambda=\hat{x}^{\lambda_0}$. Then, for $\lambda$ sufficiently close to $\lambda_0$, since $\bar{u}(x),~\bar{v}(x)$ are continuous in $\Sigma_\lambda$, it follows by (\ref{y2}) that
\begin{eqnarray}
U_\lambda(x)&=&[\bar{u}_\lambda(x)-\bar{u}_{\lambda_0}(\hat{x})]+
[\bar{u}_{\lambda_0}(\hat{x})-\bar{u}(\hat{x})]+[\bar{u}(\hat{x})-\bar{u}(x)]\nonumber\\
&\geq&0+c_0+[\bar{u}(\hat{x})-\bar{u}(x)]\nonumber\\
&\geq&0,~~x\in B_\eta((x^0)^{\lambda})\cap\mathbb{R}^n_+,
\label{y3}
\end{eqnarray}
and similarly,
\begin{equation}
V_\lambda(x)\geq0,~~x\in B_\eta((x^0)^{\lambda})\cap\mathbb{R}^n_+.
\label{mimei}
\end{equation}
(\ref{y3}) and (\ref{mimei}) indicate that $\Sigma_{U_\lambda}^-,~\Sigma_{V_\lambda}^-$ have no intersections with $B_\eta((x^0)^{\lambda})$.

By (\ref{37}), the negative minima of $U_\lambda$, $V_\lambda$ connot be attained outside of $B_{R_0}(0)$. Next we will prove that they can neither be attained inside of $B_{R_0}(0)$, i.e., for $\lambda$ sufficiently close to $\lambda_0$,
\begin{equation}
U_\lambda(x),~V_\lambda(x)\geq0,~~x\in(\Sigma_\lambda\cap B_{R_0}(0)).
\label{312}
\end{equation}

Actually, for any $\lambda\in[\lambda_0,\lambda_0+\varepsilon)$, there exists a small $\delta>0$ such that if
\begin{equation}
U_\lambda(x),~V_\lambda(x)\geq0,~~x\in\Sigma_{\lambda_0-\delta},
\label{313}
\end{equation}
then
\begin{equation}
U_\lambda(x),~V_\lambda(x)\geq0,~~x\in(\Sigma_\lambda\setminus\Sigma_{\lambda_0-\delta}).
\label{314}
\end{equation}
This can be easily obtained by Theorem \ref{thm2.2}: for $\varepsilon,\delta\ll|\lambda_0|$, it holds $(x^0)^\lambda\in H_{\lambda_0-\delta}:=\{x\in\mathbb{R}^n|x_1<\lambda_0-\delta\}$, then the lower bounds of $c_1(\tilde{x}),~c_2(\bar{x})$ in $(\Sigma_{U_\lambda}^-\cup\Sigma_{V_\lambda}^-)\subset(\Sigma_\lambda\setminus
\Sigma_{\lambda_0-\delta})$ can be seen from (\ref{3303}) and (\ref{36}), now by (\ref{313}) and Remark \ref{rem2.1}, to derive (\ref{314}) from (\ref{313}), we only need to use Theorem \ref{thm2.2} to $U_\lambda,V_\lambda$ with the narrow region $$\Omega=(\Sigma_{U_\lambda}^-\cup\Sigma_{V_\lambda}^-)\subset(\Sigma_\lambda\setminus
\Sigma_{\lambda_0-\delta}).$$
Here what we need to point out is that, if $(\Sigma_{U_\lambda}^-\cup\Sigma_{V_\lambda}^-)\neq\emptyset$, then the minimum points $\tilde{x},~\bar{x}$ can indeed be gained in the interior of $\Sigma_\lambda\setminus\Sigma_{\lambda_0-\delta}$ by virtue of
$$
\lim_{|x|\rightarrow\infty}U_\lambda(x)=\lim_{|x|\rightarrow\infty}V_\lambda(x)=0,~~~
U_\lambda(x)=V_\lambda(x)=0~\text{on} ~T_\lambda,
$$
$$
(x^0)^\lambda\in H_{\lambda_0-\delta},~~U_\lambda(x)=V_\lambda(x)=0~\text{on} ~\{x\in\partial\mathbb{R}^n_+|\lambda_0
-\delta\leq x_1\leq\lambda\},
$$
and the continuities of $U_\lambda(x),~V_\lambda(x)$ in this narrow region.

Now what left is to show (\ref{313}), and we only need to prove
\begin{equation}
U_{\lambda}(x),V_{\lambda}(x)\geq0,~~x\in(\Sigma_{\lambda_0-\delta}\cap B_{R_0}(0))\setminus B_\eta((x^0)^\lambda).
\label{315}
\end{equation}
It follows from (\ref{3100}) that there exists a constant $c>0$ such that
$$
U_{\lambda_0}(x),~V_{\lambda_0}(x)\geq c,~~x\in\overline{(\Sigma_{\lambda_0-\delta}\cap B_{R_0}(0)\cap\{x_n>\xi\}\setminus B_\eta((x^0)^{\lambda_0}))},
$$
where $\xi>0$ is a small constant. Since $U_{\lambda}(x),~V_{\lambda}(x)$ depend continuously on $\lambda$, there exists $\varepsilon>0$ and $\varepsilon<\delta$, such that for all $\lambda\in(\lambda_0,\lambda_0+\varepsilon)$, we have
\begin{equation}
U_{\lambda}(x),~V_{\lambda}(x)\geq0,~~x\in\overline{(\Sigma_{\lambda_0-\delta}\cap B_{R_0}(0)\cap\{x_n>\xi\}\setminus B_\eta((x^0)^{\lambda}))}.
\end{equation}
By the continuities of $U_\lambda(x),~V_\lambda(x)$ near $x_n=0$ and $U_\lambda(x)=V_\lambda(x)=0$ on $\partial\mathbb{R}^n_+\setminus B_\eta((x^0)^{\lambda})$, letting $\xi\rightarrow0$, we can further deduce that
\begin{equation}
U_{\lambda}(x),~V_{\lambda}(x)\geq0,~~x\in\overline{(\Sigma_{\lambda_0-\delta}\cap B_{R_0}(0)\setminus B_\eta((x^0)^{\lambda}))}.
\label{319}
\end{equation}
This verifies (\ref{315}), and thus verifies (\ref{313}).

Combining (\ref{37}), (\ref{y3}), (\ref{mimei}), (\ref{313}) and (\ref{314}), we conclude that for all $\lambda\in(\lambda_0,\lambda_0+\varepsilon)$,
\begin{equation}
U_{\lambda}(x),~V_{\lambda}(x)\geq0,~~x\in\Sigma_{\lambda},
\label{320}
\end{equation}
which contradicts the definition of $\lambda_0$. Therefore, we must have
\begin{equation}
U_{\lambda_0}(x)=V_{\lambda_0}(x)\equiv0,~~ x\in\Sigma_{\lambda_0}.
\label{321}
\end{equation}

For any point $P\in\partial\mathbb{R}^n_+$, define $l_P$ the line parallel to the $x_n$-axis and passing through $P$. Since the $x_1$ direction is arbitrarily chosen, now we know by (\ref{321}) that $\bar{u}(x),~\bar{v}(x)$ are axially symmetric about some line $l_Q$ different from $l_{x^0}$, and $\bar{u}(x),~\bar{v}(x)\not\equiv c$ are bounded near $x^0$, thus $u(x),~v(x)\not\equiv c$ are bounded in $\mathbb{R}^n$ too, and furthermore,
\begin{equation}
u(x)\sim o(\frac{1}{|x|^{n-\alpha}}),~v(x)\sim o(\frac{1}{|x|^{n-\beta}})~~\text{at ~infinity}.
\label{infi1}
\end{equation}
We also know that $(-\lap)^{\alpha/2}u(x),~(-\lap)^{\beta/2}v(x)$ are axially symmetric about the same line $l_Q$, which can be easily proved through elementary computation with the help of (\ref{_+_}). Thus the right hand side of the following equations
\begin{eqnarray}
(-\lap)^{\alpha/2} \bar{u}(x)&=&\frac{1}{|x-x^0|^{n+\alpha}}f(|x-x^0|^{n-\beta}\bar{v}(x)),\quad x\in\mathbb{R}^n_+,\nonumber\\
(-\lap)^{\beta/2} \bar{v}(x)&=&\frac{1}{|x-x^0|^{n+\beta}}g(|x-x^0|^{n-\alpha}\bar{u}(x)),\quad x\in\mathbb{R}^n_+.\nonumber
\end{eqnarray}
should have the same symmetry. From this, we are able to prove that
\begin{equation}
f(t)=C_1t^{\frac{n+\alpha}{n-\beta}},\quad t\in(0, max_{\mathbb{R}^n_+}v],
\label{ref1}
\end{equation}
and
\begin{equation}
g(t)=C_2t^{\frac{n+\beta}{n-\alpha}},\quad t\in(0, max_{\mathbb{R}^n_+}u],
\label{eref1}
\end{equation}
for some positive constants $C_1,~C_2$ ($C_1,~C_2$ cannot be $0$, otherwise $u(x)=u(x_n),~v(x)=v(x_n)$ are axially symmetric about $l_{x^0}$, please refer to \cite{ZLCC}).

The proofs of (\ref{ref1}) and (\ref{eref1}) are sophisticated, we have proved them before, please refer to \cite{ZY}.

Now let us admit

\textbf{Claim 3.3.} \textsl{If (\ref{infi1})-(\ref{eref1}) hold, then (\ref{0}) has no positive solutions in $\mathbb{R}^n_+$.}

The proof of Claim 3.3 is long and sophisticated, we put it in Section 4. Till now we have shown that (\ref{0}) possesses no positive solutions under \textbf{Possibility (i)}.

\textbf{Possibility (\textbf{ii}).} From the definition of $\lambda_0$, one knows that
\begin{equation}
\lambda_0=x_1^0~\text{and}~U_{\lambda_0},V_{\lambda_0}\geq0,~~x\in\Sigma_{\lambda_0}.
\label{327}
\end{equation}
Now we move the plane $T_\lambda$ from $+\infty$ to the left. Similarly, we can also derive that either
\begin{equation}
\lambda_0>x_1^0 ~\text{and}~U_{\lambda_0}=V_{\lambda_0}\equiv0,~~x\in\Sigma_{\lambda_0},
\qquad\quad\quad\qquad
\label{329}
\end{equation}
or
\begin{equation}
\lambda_0=x_1^0 ~\text{and}~U_{\lambda_0},V_{\lambda_0}\leq0,~~x\in\Sigma_{\lambda_0}.\qquad\quad\quad\qquad
\label{328}
\end{equation}
The case described by $(\ref{329})$ can be handled with the same way as \textbf{Possibility (\textbf{i})}. Now from (\ref{327}) and (\ref{328}), we have
$$
\lambda_0=x_1^0 ~\text{and}~U_{\lambda_0}=V_{\lambda_0}\equiv0,~x\in\Sigma_{\lambda_0}.
$$

So far, we have proved that $\bar{u},~\bar{v}$ are symmetric about the plane $T_{x^0_1}$. Since the $x_1$ direction is arbitrarily chosen, we have actually shown that $\bar{u},~\bar{v}$ are axially symmetric about $l_{x^0}$. For any two points $X^1,X^2\in\mathbb{R}^n$ satisfying $X^1_n=X^2_n$, denote the orthogonal projections of $X^1,X^2$ on $\partial\mathbb{R}^n_+$ by $\hat{X}^1,\hat{X}^2$ respectively, choose $x^0=\frac{\hat{X}^1+\hat{X}^2}{2}$. Since $\bar{u},\bar{v}$ are axially symmetric about $l_{x^0}$, so are $u,v$, hence $u(X^1)=u(X^2)$ and $v(X^1)=v(X^2)$. This implies that $u=u(x_n)$, $v=v(x_n)$.

Next, we show that $u=u(x_n),v=v(x_n)$, $u,v>0$ in $\mathbb{R}^n_+$ contradict the finiteness of $u,v$ respectively, which indicates that (\ref{0}) possesses no positive solutions. Let us first admit the following claim, its proof is pretty sophisticated and will be given in Section 4.

\textbf{Claim 3.4.} \textsl{Suppose that $f,g,u$ and $v$ satisfy the conditions in Theorem \ref{thm1}, $u,v>0$ in $\mathbb{R}^n_+$. Then $u=u(x_n),v=v(x_n)$ contradict the finiteness of $u,v$ respectively.}

By (\ref{equal0}), (\ref{larger0}) and Claim 3.4, we can immediately conclude that the only nonnegative solutions of (\ref{0}) have to be $u=v\equiv0$ in $\mathbb{R}^n$, and thus finish the proof of Theorem \ref{thm1}.

\end{proof}

\section{Proofs of Four Claims}
Here we prove Claims 3.1, 3.2, 3.3 and 3.4. Without loss of generality, we may suppose the center of the Kelvin transform $x^0=0$.
\subsection{Proof of Claim 3.1}.
\begin{proof} Please note that we already assumed $u,v>0$ in $\mathbb{R}^n_+$ in the beginning of Section 3, so it holds $\bar{u},\bar{v}>0$ here. Now under this assumption, we first show that $$U_\lambda(x)\geq c_\lambda,~~~~x\in B_\epsilon(0^\lambda)\cap\mathbb{R}^n_+.$$ Since $\bar{u}(x)\rightarrow0$ as $|x|\rightarrow\infty$, and $(B_\epsilon(0^\lambda)\cap\mathbb{R}^n_+)\subset\Sigma_\lambda$ when $\lambda$ is sufficiently negative, we only need to prove $\bar{u}_\lambda(x)\geq2c_\lambda$ in $B_\epsilon(0^\lambda)\cap\mathbb{R}^n_+$, i.e. $$\bar{u}(x)\geq2c_\lambda ~\text{in}~ B_\epsilon(0)\cap\mathbb{R}^n_+.$$
Using
$$
u(x)=\frac{1}{|x|^{n-\alpha}}\bar{u}(\frac{x}{|x|^2}),
$$
we only need to prove
$$
u(x)\geq2c_\lambda\frac{1}{|x|^{n-\alpha}},~~~\text{when}~x\in\mathbb{R}^n_+ ~\text{and} ~|x|~\text{sufficiently large.}
$$

Let
$$
\varphi(x)=c_{n,-\alpha}\int_{\mathbb{R}^n_+}\frac{\eta(y)f(v(y))}{|x-y|^{n-\alpha}}dy,
$$
where $c_{n,-\alpha}>0$ is a proper constant, $\eta(y)\in C^\infty(\mathbb{R}^n_+)$ is a cutoff function and
\begin{equation}
\eta(y)=\left\{\begin{array}{ll}
1, & \qquad y\in B_1(0)\cap\{x_n>\frac{1}{2}\}, \\
0, & \qquad y\notin B_2(0)\cap\{x_n>0\}.
\end{array}\right.
\label{jbb2}\end{equation}
Then
\begin{eqnarray}
(-\lap)^{\alpha/2}\varphi(x)&=&c_{n,-\alpha}\int_{\mathbb{R}^n_+}(-\lap)^{\frac{\alpha}{2}}(\frac{1}{|x-y|^{n-\alpha}})\eta(y)f(v(y))dy\nonumber\\
&=&\eta(x)f(v(x)),~~~\text{in}~\mathbb{R}^n_+.
\end{eqnarray}
Hence
\begin{eqnarray}
(-\lap)^{\alpha/2}(u-\varphi)&=&f(v)-\eta f(v)\nonumber\\
&=&f(v)(1-\eta)\nonumber\\
&\geq&0,~~~\text{in}~\mathbb{R}^n_+.
\label{APp1}
\end{eqnarray}

For any $x\in\mathbb{R}^n_+\setminus B_R(0)$, we have
$$
c_{n,-\alpha}\int_{B_1(0)\cap\{y_n>\frac{1}{2}\}}\frac{f(v(y))}{|x-y|^{n-\alpha}}dy\leq\varphi(x)\leq c_{n,-\alpha}\int_{B_2(0)\cap\{y_n>0\}}\frac{f(v(y))}{|x-y|^{n-\alpha}}dy.
$$
Because $v\in C(\mathbb{R}^n)$, $f\geq0$ is strictly increasing for $t\geq0$, it yields that for any $x\in\mathbb{R}^n_+\setminus B_R(0)$ and $R$ large enough, there exist two positive constants $C_1$ and $C_2$ such that
\begin{equation}
\frac{C_1}{|x|^{n-\alpha}}\leq\varphi(x)\leq\frac{C_2}{|x|^{n-\alpha}}\leq\frac{C_2}{R^{n-\alpha}}.
\label{App3}
\end{equation}
From (\ref{APp1}) and (\ref{App3}), we get
\begin{equation}
\left\{\begin{array}{ll}
(-\lap)^{\alpha/2}(u-\varphi+\frac{C_2}{R^{n-\alpha}})\geq0, & \qquad \text{in}~ \mathbb{R}^n_+\cap B_R(0), \\
u-\varphi+\frac{C_2}{R^{n-\alpha}}\geq0, & \qquad \text{in}~ \mathbb{R}^n_+\setminus B_R(0).
\end{array}\right.
\label{jbb2}\end{equation}
Noting Theorem 2.1, it gets
$$
u-\varphi+\frac{C_2}{R^{n-\alpha}}\geq0,~~\text{in}~ \mathbb{R}^n_+\cap B_R(0),
$$
together with (\ref{jbb2}), letting $R\rightarrow\infty$, we derive
$$
u\geq\varphi,~~~x\in\mathbb{R}^n_+,
$$
and by (\ref{App3}),
$$
u(x)\geq\frac{2c_\lambda}{|x|^{n-\alpha}},~~~\text{for}~x\in\mathbb{R}^n_+~\text{and}~|x|~\text{sufficiently large}.
$$

Similarly, one can also obtain that
$$V_\lambda(x)\geq c_\lambda,~~~~x\in B_\epsilon(0^\lambda)\cap\mathbb{R}^n_+.$$
This completes the proof of Claim 3.1.

\end{proof}

\subsection{Proof of Claim 3.2}

Before proving Claim 3.2, we first narrate two propositions.
\begin{proposition}
~~Assume that $u\in C_{loc}^{1,1}(\mathbb{R}^n_+)\cap\mathcal{L}_\alpha(\mathbb{R}^n_+)$, $v\in C_{loc}^{1,1}(\mathbb{R}^n_+)\cap\mathcal{L}_\beta(\mathbb{R}^n_+)$ are locally bounded positive solutions to the problem
\begin{equation}
\left\{\begin{array}{lll}
(-\lap)^{\alpha/2} u(x)=f(v(x)), &  \\
(-\lap)^{\beta/2} v(x)=g(u(x)), & \qquad x\in\mathbb{R}^n_+,\\
u(x),v(x)\equiv0, & \qquad x\notin\mathbb{R}^n_+,
\end{array}\right.
\label{423}
\end{equation}
then they are also solutions to
\begin{equation}
\left\{\begin{array}{lll}
u(x)=\int_{\mathbb{R}^n_+}G^\alpha_\infty(x,y)f(v(y))dy, &  \\
v(x)=\int_{\mathbb{R}^n_+}G^\beta_\infty(x,y)g(u(y))dy, & \qquad x\in\mathbb{R}^n_+,\\
u,v\equiv0, & \qquad x\not\in\mathbb{R}^n_+,
\end{array}\right.\label{A424}
\end{equation}
and vice versa. Here $G^\alpha_\infty(x,y)$, $G^\beta_\infty(x,y)$ are the Green functions of (\ref{423}):
$$
G^\alpha_\infty(x,y)=\frac{A_{n,\alpha}}{s^{\frac{n-\alpha}{2}}}\left[1-\frac{B_{n,\alpha}}{(t+s)^{\frac{n-2}{2}}}
\int_0^{\frac{s}{t}}\frac{(s-tb)^{\frac{n-2}{2}}}{b^{\alpha/2}(1+b)}db\right],~~x,y\in\mathbb{R}^n_+,
$$
$$
G^\beta_\infty(x,y)=\frac{A_{n,\beta}}{s^{\frac{n-\beta}{2}}}\left[1-\frac{B_{n,\beta}}{(t+s)^{\frac{n-2}{2}}}
\int_0^{\frac{s}{t}}\frac{(s-tb)^{\frac{n-2}{2}}}{b^{\beta/2}(1+b)}db\right],~~x,y\in\mathbb{R}^n_+,
$$
where $t=4x_ny_n,~s=|x-y|^2$ and
$A_{n,\alpha},B_{n,\alpha}~(A_{n,\beta},B_{n,\beta})$ are positive constants which only depend on $n,\alpha~(n,\beta)$.
\label{pro4.123}
\end{proposition}

The proof of Proposition \ref{pro4.123} here is entirely similar to the proof of Theorem 2.1 in \cite{ZC}. Since the conditions ``$f(t),g(t)\geq0$ are strictly increasing for $t\geq0$" and ``$u,v\in C_{loc}^{1,1}(\mathbb{R}^n_+)$" guarantee that: ($i$) $f(v(x))$ and $g(u(x))$ are locally bounded on $\mathbb{R}^n_+$; ($ii$) $f(t),g(t)\geq C_0$ for $t>R$, where $R>0$ is sufficiently large and $C_0$ is a positive constant. we only need to substitute $f(v(x))$ (and $g(u(x))$) for $x_n^\gamma u^p(x)$ in the proof of Theorem 2.1 in \cite{ZC}. Here we omit the details.

\begin{proposition} (\cite{CFY})
~~For any $x,y\in\mathbb{R}^n_+$, it holds that
\begin{equation}
\frac{\partial G_\infty^\alpha(x,y)}{\partial s}<0,~~\frac{\partial G_\infty^\beta(x,y)}{\partial s}<0,
\label{4230}
\end{equation}
where $s=|x-y|^2$. Thus for any $x,y\in\Sigma_\lambda$,
\begin{equation}
G_\infty^\alpha(x,y)-G_\infty^\alpha(x,y^\lambda)>0,~~G_\infty^\beta(x,y)
-G_\infty^\beta(x,y^\lambda)>0.
\label{42301}
\end{equation}
\label{pro4.1230}
\end{proposition}

The proof of Proposition \ref{pro4.1230} is standard and can be found in the proof of Lemma 2.1 in \cite{CFY}. Here we omit the details.

\begin{proof} From (\ref{3100}) and the continuities of $U_{\lambda_0}$ and $V_{\lambda_0}$, there exists a point $x^1\in\Sigma_{\lambda_0}$ and a small positive $\delta$ such that
\begin{equation}
C_1\leq U_{\lambda_0}(x),V_{\lambda_0}(x)\leq C_2,~~x\in \overline{B_{\delta/2}(x^1)}\subset\Sigma_{\lambda_0}.
\label{AA1}
\end{equation}
By Proposition \ref{pro4.123} and (\ref{3002}) (please note that in this section we already let $x^0=0$ for simplicity), we have
\begin{equation}
(-\lap)^{\alpha/2} \bar{u}(x)=\int_{\mathbb{R}^n_+}G^\alpha_\infty(x,y)\frac{1}{|y|^{n+\alpha}}
f(|y|^{n-\beta}\bar{v}(y))dy,\quad x\in\mathbb{R}^n_+.
\label{300200}
\end{equation}
Through elementary computation, for any $x\in\mathbb{R}^n_+$, we deduce that 
\begin{equation}
U_{\lambda_0}(x)=C\int_{\Sigma_{\lambda_0}}(G^\alpha_\infty(x,y)-G^\alpha_\infty(x,y^{\lambda_0}))
(\frac{f(|y^{\lambda_0}|^{n-\beta}\bar{v}(y^{\lambda_0}))}{|y^{\lambda_0}|^{n+\alpha}}-
\frac{f(|y|^{n-\beta}\bar{v}(y))}{|y|^{n+\alpha}})dy.
\end{equation}
It follows from the monotonicities of $f(t)$ and $f(t)/t^{\frac{n+\alpha}{n-\beta}}$ that for any $x\in B_\epsilon(0^{\lambda_0})\cap\mathbb{R}^n_+$,
\begin{eqnarray}
U_{\lambda_0}(x)&=&C\int_{\Sigma_{\lambda_0}}(G^\alpha_\infty(x,y)-G^\alpha_\infty(x,y^{\lambda_0}))
(\frac{f(|y^{\lambda_0}|^{n-\beta}\bar{v}(y^{\lambda_0}))}{|y^{\lambda_0}|^{n+\alpha}}-
\frac{f(|y|^{n-\beta}\bar{v}(y))}{|y|^{n+\alpha}})dy\nonumber\\
&\geq&C\int_{\Sigma_{\lambda_0}}(G^\alpha_\infty(x,y)-G^\alpha_\infty(x,y^{\lambda_0}))
(\frac{f(|y|^{n-\beta}\bar{v}(y^{\lambda_0}))-f(|y|^{n-\beta}\bar{v}(y))}
{|y|^{n+\alpha}})dy\nonumber\\
&\geq&C\int_{B_{\delta/2}(x^1)}(G^\alpha_\infty(x,y)-G^\alpha_\infty(x,y^{\lambda_0}))
(\frac{f(|y|^{n-\beta}\bar{v}(y^{\lambda_0}))-f(|y|^{n-\beta}\bar{v}(y))}
{|y|^{n+\alpha}})dy.\nonumber\\
\label{App7}
\end{eqnarray}
Since
$$
|y|^{n-\beta}\bar{v}(y^{\lambda_0})-|y|^{n-\beta}\bar{v}(y)=
|y|^{n-\beta}V_{\lambda_0}(y)~~\text{in}~\overline{B_{\delta/2}(x^1)},
$$
one knows by (\ref{AA1}) and the continuity of $\bar{v}$ that
\begin{equation}
C_3\leq|y|^{n-\beta}\bar{v}(y^{\lambda_0})-|y|^{n-\beta}\bar{v}(y)\leq C_4~~\text{in}~\overline{B_{\delta/2}(x^1)}
\label{eq4.14}
\end{equation}
for some constants $C_3, C_4>0$, $C_3<C_4$. Because we already supposed $u,v>0$ in $\mathbb{R}^n_+$ at the beginning of Section 3, so it also holds $\bar{u},\bar{v}>0$ in $\mathbb{R}^n_+$, together with the continuity of $\bar{v}$, one gets 
\begin{equation}
C_5\leq|y|^{n-\beta}\bar{v}(y^{\lambda_0}),|y|^{n-\beta}\bar{v}(y)\leq C_6~~\text{in}~\overline{B_{\delta/2}(x^1)}
\label{eq4.15}
\end{equation}
for some constants $C_5, C_6>0$, $C_5<C_6$. Since $f$ is strictly increasing in $[0,+\infty)$, by (\ref{eq4.14}) and (\ref{eq4.15}) we obtain that, there exists $C_7,~C_8>0$, $C_7<C_8$ such that for any $y\in\overline{B_{\delta/2}(x^1)}$,
\begin{equation}
C_7\leq f(|y|^{n-\beta}\bar{v}(y^{\lambda_0}))-f(|y|^{n-\beta}\bar{v}(y))\leq C_8.
\label{eq4.16}
\end{equation}
Now (\ref{eq4.14}) and (\ref{eq4.16}) indicate that 
\begin{eqnarray*}
f(|y|^{n-\alpha}\bar{v}(y^{\lambda_0}))-f(|y|^{n-\beta}\bar{v}(y))&\geq& C_0(|y|^{n-\beta}\bar{v}(y^{\lambda_0})-|y|^{n-\beta}\bar{v}(y))\\
&=&C_0(|y|^{n-\beta}V_{\lambda_0}(y)),~~~y\in\overline{B_{\delta/2}(x^1)}
\end{eqnarray*}
for some $C_0>0$. Hence, it shows by (\ref{AA1}), (\ref{App7}) and Proposition \ref{pro4.1230} that
\begin{eqnarray}
U_{\lambda_0}(x)&\geq&C\int_{B_{\delta/2}(x^1)}
(G^\alpha_\infty(x,y)-G^\alpha_\infty(x,y^{\lambda_0}))
(\frac{f(|y|^{n-\beta}\bar{v}(y^{\lambda_0}))-f(|y|^{n-\beta}\bar{v}(y))}{|y|^{n+\alpha}})dy\nonumber\\
&\geq&C\int_{B_{\delta/2}(x^1)}(G^\alpha_\infty(x,y)-G^\alpha_\infty(x,y^{\lambda_0}))\frac{C_0}{|y|^{\alpha+\beta}}V_{\lambda_0}(y)dy\nonumber\\
&\geq&C\int_{B_{\delta/2}(x^1)}C_0dy\nonumber\\
&:=&2c_0,~~~~x\in B_\epsilon(0^{\lambda_0})\cap\mathbb{R}^n_+.
\label{App8}
\end{eqnarray}

Similarly, one can also derive that $$V_{\lambda_0}(x)\geq2c_0,~~x\in B_\epsilon(0^{\lambda_0})\cap\mathbb{R}^n_+.$$ Now Claim 3.2 is proved.

\end{proof}

\subsection{Proof of Claim 3.3}

This subsection is dedicated to the proof of Claim 3.3. Redefine
$$
\hat{T}_\lambda=\{x=(x',x_n)|x_n=\lambda~\text{for~some~}\lambda>0,~
\lambda\in\mathbb{R}\}
$$
with $x'=(x_1,x_2,\ldots,x_{n-1}),$ let
$$
x^\lambda=(x_1,x_2,\ldots,x_{n-1},2\lambda-x_n)
$$
be the refection of $x$ about the plane $\hat{T}_\lambda$. Denote
$$
\Sigma_\lambda^*=\{x\in\mathbb{R}^n_+|x_n<\lambda\},~~~~\tilde{\Sigma}_\lambda^*=
\{x|x^\lambda\in\Sigma_\lambda^*\},~~~~\Sigma_\lambda^C=\mathbb{R}^n_+
\setminus\tilde{\Sigma}_\lambda^*,
$$
$$
u_\lambda(x):=u(x^\lambda),~~v_\lambda(x):=v(x^\lambda),
$$
and
$$
U_\lambda(x):=u_\lambda(x)-u(x),~~V_\lambda(x):=v_\lambda(x)-v(x).
$$

Set
$$
\Sigma_{U_\lambda}^-=\{x\in\Sigma_\lambda^*|U_\lambda(x)<0\},~~
\Sigma_{V_\lambda}^-=\{x\in\Sigma_\lambda^*|V_\lambda(x)<0\},
$$
and
$$
\Sigma^-_\lambda=\Sigma_{U_\lambda}^-\cup\Sigma_{V_\lambda}^-.
$$

From (\ref{infi1})-(\ref{eref1}), (\ref{0}) becomes
\begin{equation}
\left\{\begin{array}{lll}
(-\lap)^{\alpha/2} u(x)=C_1v^{p_0}(x), &  \\
(-\lap)^{\beta/2} v(x)=C_2u^{q_0}(x), & \qquad x\in\mathbb{R}^n_+,\\
u(x),v(x)\equiv0, & \qquad x\not\in\mathbb{R}^n_+,
\end{array}\right.
\label{eq4.6}
\end{equation}
where $p_0=\frac{n+\alpha}{n-\beta},~q_0=\frac{n+\beta}{n-\alpha}$. By (\ref{ref1}), (\ref{eref1}) and Proposition \ref{pro4.123} we know that if $u,v\in C(\mathbb{R}^n)$ are solutions to
(\ref{eq4.6}), then they are also solutions of
\begin{equation}
\left\{\begin{array}{lll}
u(x)=C_1\int_{\mathbb{R}^n_+}G^\alpha_\infty(x,y)v^{p_0}(y)dy, &  \\
v(x)=C_2\int_{\mathbb{R}^n_+}G^\beta_\infty(x,y)u^{q_0}(y)dy, & \qquad x\in\mathbb{R}^n_+,\\
u,v\equiv0, & \qquad x\not\in\mathbb{R}^n_+,
\end{array}\right.\label{A4.7}
\end{equation}
and vice versa. Hence, to prove Claim 3.3, we only need to prove that (\ref{A4.7}) possesses no positive solutions which satisfy (\ref{infi1}), and without loss of generality, we may suppose $C_1=C_2=1$ for simplicity.

Before the proof, let us first narrate three key ingredients which will be used in the forthcoming integral estimate.

\begin{proposition} (\cite{CFY}, \textsl{An equivalent form of the Hardy-Littlewood-Soblev inequality})~~Assume $0<\alpha,\beta<n$ and $\Omega\subset\mathbb{R}^n$, let $\varphi(x)\in L^{\frac{nr}{n+\alpha r}}(\Omega)$, $\psi(x)\in L^{\frac{nr}{n+\beta r}}(\Omega)$ for $\max\{\frac{n}{n-\alpha},\frac{n}{n-\beta}\}<r<\infty$. Define
$$
T\varphi(x):=\int_\Omega\frac{1}{|x-y|^{n-\alpha}}\varphi(y)dy,~~~W\psi(x)
:=\int_\Omega\frac{1}{|x-y|^{n-\beta}}\psi(y)dy.
$$
Then
$$
\|T\varphi\|_{L^r(\Omega)}\leq C(n,r,\alpha)\| \varphi\|_{L^{\frac{nr}{n+\alpha r}}(\Omega)},~~~
\|W\psi\|_{L^r(\Omega)}\leq C(n,r,\beta)\| \psi\|_{L^{\frac{nr}{n+\beta r}}(\Omega)}.
$$
\label{propos4.1230}
\end{proposition}

\begin{proposition} (\cite{CFY})
~~For any $x,y\in\mathbb{R}^n_+$, it holds
\begin{equation}
\frac{\partial G_\infty^\alpha(x,y)}{\partial s}<0,~~\frac{\partial G_\infty^\beta(x,y)}{\partial s}<0,~~~\frac{\partial G_\infty^\alpha(x,y)}{\partial t}>0,~~\frac{\partial G_\infty^\beta(x,y)}{\partial t}>0,
\label{423eq0}
\end{equation}
where $s=|x-y|^2,~t=4x_ny_n$. Thus 
\begin{description}
\item[~~~(i)]for any $x,y\in\Sigma_\lambda^*$, $x\neq y$,
$$
G^\alpha_\infty(x^\lambda,y^\lambda)>\max\{G^\alpha_\infty(x^\lambda,y),
G^\alpha_\infty(x,y^\lambda)\},
$$
$$
G^\beta_\infty(x^\lambda,y^\lambda)>\max\{G^\beta_\infty(x^\lambda,y),
G^\beta_\infty(x,y^\lambda)\};
$$
\item[~~~(ii)]for any $x,y\in\Sigma_\lambda^*$, $x\neq y$,
$$
G^\alpha_\infty(x^\lambda,y^\lambda)-G^\alpha_\infty(x,y)>| G^\alpha_\infty(x^\lambda,y)-
G^\alpha_\infty(x,y^\lambda)|,
$$
$$
G^\beta_\infty(x^\lambda,y^\lambda)-G^\beta_\infty(x,y)>| G^\beta_\infty(x^\lambda,y)-
G^\beta_\infty(x,y^\lambda)|;
$$
\item[~~~(iii)]for any $x\in\Sigma_\lambda^*$, $y\in\Sigma_\lambda^C$,
\begin{equation}
G^\alpha_\infty(x^\lambda,y)>G^\alpha_\infty(x,y),~~~
G^\beta_\infty(x^\lambda,y)>G^\beta_\infty(x,y).
\label{eqwua4.20}
\end{equation}
\end{description}
\label{propo4.1230}
\end{proposition}

The proof of Proposition \ref{propo4.1230} is standard and can be found in the proof of Lemma 2.1 in \cite{CFY}. Here we omit the details.

\begin{proposition}
~~For any $x\in\Sigma_\lambda^*$, it holds that
\begin{equation}
u(x)-u_\lambda(x)\leq\int_{\Sigma_\lambda^*}\left[G_\infty^\alpha(x^\lambda,y^\lambda)
-G_\infty^\alpha(x,y^\lambda)\right][v^{p_0}(y)-v_\lambda^{p_0}(y)]dy,
\label{eqw4.20}
\end{equation}
\begin{equation}
v(x)-v_\lambda(x)\leq\int_{\Sigma_\lambda^*}\left[G_\infty^\beta(x^\lambda,y^\lambda)
-G_\infty^\beta(x,y^\lambda)\right][u^{q_0}(y)-u_\lambda^{q_0}(y)]dy,
\label{eqw4.200}
\end{equation}
and 
\begin{eqnarray}
u_\lambda(x)-u(x)&\geq&\int_{\Sigma_\lambda^*}\left[G_\infty^\alpha(x^\lambda,y^\lambda)
-G_\infty^\alpha(x,y^\lambda)\right][v_\lambda^{p_0}(y)-v^{p_0}(y)]dy\nonumber\\&&+
\int_{\Sigma_\lambda^C\setminus\tilde{\Sigma}_\lambda^*}\left[G_\infty^\alpha(x^\lambda,y)
-G_\infty^\alpha(x,y)\right]v^{p_0}(y)dy,
\label{eqw4.201}
\end{eqnarray}
\begin{eqnarray}
v_\lambda(x)-v(x)&\geq&\int_{\Sigma_\beta^*}\left[G_\infty^\beta(x^\lambda,y^\lambda)
-G_\infty^\beta(x,y^\lambda)\right][u_\lambda^{q_0}(y)-u^{q_0}(y)]dy\nonumber\\&&+
\int_{\Sigma_\lambda^C\setminus\tilde{\Sigma}_\lambda^*}\left[G_\infty^\beta(x^\lambda,y)
-G_\infty^\beta(x,y)\right]u^{q_0}(y)dy.
\label{eqw4.202}
\end{eqnarray}
\label{proposi4.1230}
\end{proposition}

\textbf{Proof of Proposition \ref{proposi4.1230}.}~~
Let us first prove (\ref{eqw4.20}). Since 
\begin{eqnarray*}
u(x)&=&\int_{\Sigma_\lambda^*}G_\infty^\alpha(x,y)v^{p_0}(y)dy+
\int_{\tilde{\Sigma}_\lambda^*}G_\infty^\alpha(x,y)v^{p_0}(y)dy+
\int_{\Sigma_\lambda^C\setminus\tilde{\Sigma}_\lambda^*}
G_\infty^\alpha(x,y)v^{p_0}(y)dy\nonumber\\
&=&\int_{\Sigma_\lambda^*}G_\infty^\alpha(x,y)v^{p_0}(y)dy+
\int_{\Sigma_\lambda^*}G_\infty^\alpha(x,y^\lambda)v_\lambda^{p_0}(y)dy+
\int_{\Sigma_\lambda^C\setminus\tilde{\Sigma}_\lambda^*}
G_\infty^\alpha(x,y)v^{p_0}(y)dy,
\label{eqe5.34}
\end{eqnarray*}
and
\begin{eqnarray*}
u_\lambda(x)&=&\int_{\Sigma_\lambda^*}G_\infty^\alpha(x^\lambda,y)v^{p_0}(y)dy+
\int_{\tilde{\Sigma}_\lambda^*}G_\infty^\alpha(x^\lambda,y)v^{p_0}(y)dy+
\int_{\Sigma_\lambda^C\setminus\tilde{\Sigma}_\lambda^*}
G_\infty^\alpha(x^\lambda,y)v^{p_0}(y)dy\nonumber\\
&=&\int_{\Sigma_\lambda^*}G_\infty^\alpha(x^\lambda,y)v^{p_0}(y)dy+
\int_{\Sigma_\lambda^*}G_\infty^\alpha(x^\lambda,y^\lambda)v_\lambda^{p_0}(y)dy+
\int_{\Sigma_\lambda^C\setminus\tilde{\Sigma}_\lambda^*}
G_\infty^\alpha(x^\lambda,y)v^{p_0}(y)dy,
\label{eqe5.344}
\end{eqnarray*}
by Proposition \ref{propo4.1230}, we arrive at 
\begin{eqnarray*}
&&u(x)-u_\lambda(x)\nonumber\\&=&\int_{\Sigma_\lambda^*}\left[G_\infty^\alpha(x,y)
-G_\infty^\alpha(x^\lambda,y)\right]v^{p_0}(y)dy+
\int_{\Sigma_\lambda^*}\left[G_\infty^\alpha(x,y^\lambda)
-G_\infty^\alpha(x^\lambda,y^\lambda)\right]v_\lambda^{p_0}(y)dy\nonumber\\
&&+\int_{\Sigma_\lambda^C\setminus\tilde{\Sigma}_\lambda^*}
\left[G_\infty^\alpha(x,y)
-G_\infty^\alpha(x^\lambda,y)\right]v^{p_0}(y)dy\nonumber\\
&\leq&\int_{\Sigma_\lambda^*}\left[G_\infty^\alpha(x,y)
-G_\infty^\alpha(x^\lambda,y)\right]v^{p_0}(y)dy-
\int_{\Sigma_\lambda^*}\left[G_\infty^\alpha(x^\lambda,y^\lambda)-G_\infty^\alpha(x,y^\lambda)
\right]v_\lambda^{p_0}(y)dy\nonumber\\
&\leq&\int_{\Sigma_\lambda^*}\left[G_\infty^\alpha(x^\lambda,y^\lambda)-G_\infty^\alpha(x,y^\lambda)
\right]v^{p_0}(y)dy-
\int_{\Sigma_\lambda^*}\left[G_\infty^\alpha(x^\lambda,y^\lambda)-G_\infty^\alpha(x,y^\lambda)
\right]v_\lambda^{p_0}(y)dy\nonumber\\
&=&\int_{\Sigma_\lambda^*}\left[G_\infty^\alpha(x^\lambda,y^\lambda)-G_\infty^\alpha(x,y^\lambda)
\right][v^{p_0}(y)-v_\lambda^{p_0}(y)]dy.
\label{eqe5.345}
\end{eqnarray*}

This completes the proof of (\ref{eqw4.20}), and similarly, one can also prove (\ref{eqw4.200}).
 
Next, we prove (\ref{eqw4.201}). By Proposition \ref{propo4.1230}, we have
\begin{eqnarray*}
&&u_\lambda(x)-u(x)\nonumber\\&=&\int_{\Sigma_\lambda^*}\left[G_\infty^\alpha(x^\lambda,y)
-G_\infty^\alpha(x,y)\right]v^{p_0}(y)dy+
\int_{\tilde{\Sigma}_\lambda^*}\left[G_\infty^\alpha(x^\lambda,y)
-G_\infty^\alpha(x,y)\right]v^{p_0}(y)dy\nonumber\\
&&+\int_{\Sigma_\lambda^C\setminus\tilde{\Sigma}_\lambda^*}
\left[G_\infty^\alpha(x^\lambda,y)
-G_\infty^\alpha(x,y)\right]v^{p_0}(y)dy\nonumber\\
&\geq&\int_{\Sigma_\lambda^*}\left[G_\infty^\alpha(x,y^\lambda)
-G_\infty^\alpha(x^\lambda,y^\lambda)\right]v^{p_0}(y)dy+
\int_{\Sigma_\lambda^*}\left[G_\infty^\alpha(x^\lambda,y^\lambda)-G_\infty^\alpha(x,y^\lambda)
\right]v_\lambda^{p_0}(y)dy\nonumber\\
&&+\int_{\Sigma_\lambda^C\setminus\tilde{\Sigma}_\lambda^*}
\left[G_\infty^\alpha(x^\lambda,y)
-G_\infty^\alpha(x,y)\right]v^{p_0}(y)dy\nonumber\\
&\geq&\int_{\Sigma_\lambda^*}\left[G_\infty^\alpha(x^\lambda,y^\lambda)
+G_\infty^\alpha(x,y^\lambda)\right][v_\lambda^{p_0}(y)-v^{p_0}(y)]dy\nonumber\\
&&+\int_{\Sigma_\lambda^C\setminus\tilde{\Sigma}_\lambda^*}
\left[G_\infty^\alpha(x^\lambda,y)
-G_\infty^\alpha(x,y)\right]v^{p_0}(y)dy.
\label{eqe5.3451}
\end{eqnarray*}

This verifies (\ref{eqw4.201}). (\ref{eqw4.202}) can be similarly proved, here we omit the details.

Now let us begin to prove that (\ref{A4.7}) possesses no positive solutions. 

\begin{proof}
To get the desired nonexistence result, we apply the method of moving planes in integral forms and divide the proof into two steps. In the first step, we start from the very low end of $\mathbb{R}^n_+$, we will prove that for $\lambda>0$ sufficiently small,
\begin{equation}
U_\lambda(x),V_\lambda(x)\geq0,~~~x\in\Sigma_\lambda^*.
\label{6.01}
\end{equation}
In the second step, we will move our plane in the positive $x_n$ direction as long as (\ref{6.01}) holds to show that $U_\lambda(x),V_\lambda(x)$ are monotone increasing in $x_n$ and thus derive a contradiction.

\textsl{Step 1.} In this step, we show that for $\lambda>0$ sufficiently small, $\Sigma_\lambda^-=\emptyset$. Here we first point out that if $\Sigma_\lambda^-\neq\emptyset$, then neither $\Sigma^-_{U_\lambda}$ nor $\Sigma^-_{V_\lambda}$ is empty. To see this, without loss of generality, we may suppose $\Sigma^-_{U_\lambda}\neq\emptyset$, now for any $x\in\Sigma^-_{U_\lambda}$, by (\ref{eqw4.20}), if $\Sigma^-_{V_\lambda}=\emptyset$, then
\begin{equation}
0<u(x)-u_\lambda(x)\leq\int_{\Sigma_\lambda^*}\left[G_\infty^\alpha(x^\lambda,y^\lambda)
-G_\infty^\alpha(x,y^\lambda)\right][v^{p_0}(y)-v_\lambda^{p_0}(y)]dy\leq0.
\label{6.02}
\end{equation}
This is a contradiction. Hence if $\Sigma_\lambda^-\neq\emptyset$, then $\Sigma^-_{U_\lambda}\neq\emptyset$ and $\Sigma^-_{V_\lambda}\neq\emptyset$.
 
Now let us begin to verify (\ref{6.01}) for sufficiently small $\lambda>0$, we will prove this by the contrary. If $\Sigma_\lambda^-\neq\emptyset$, then for any $x\in\Sigma^-_{U_\lambda}$, by (\ref{eqw4.20}) and the Mean Value Theorem, we get 
\begin{eqnarray*}
0&<&u(x)-u_\lambda(x)\nonumber\\&\leq&\int_{\Sigma_\lambda^*}\left[G_\infty^\alpha(x^\lambda,y^\lambda)
-G_\infty^\alpha(x,y^\lambda)\right][v^{p_0}(y)-v_\lambda^{p_0}(y)]dy\nonumber\\
&=&\int_{\Sigma_{V_\lambda}^-}\left[G_\infty^\alpha(x^\lambda,y^\lambda)
-G_\infty^\alpha(x,y^\lambda)\right][v^{p_0}(y)-v_\lambda^{p_0}(y)]dy\nonumber\\
&&+\int_{\Sigma_\lambda^*\setminus\Sigma_{V_\lambda}^-}\left[G_\infty^\alpha(x^\lambda,y^\lambda)
-G_\infty^\alpha(x,y^\lambda)\right][v^{p_0}(y)-v_\lambda^{p_0}(y)]dy\nonumber\\
&\leq&\int_{\Sigma_{V_\lambda}^-}\left[G_\infty^\alpha(x^\lambda,y^\lambda)
-G_\infty^\alpha(x,y^\lambda)\right]
[v^{p_0}(y)-v_\lambda^{p_0}(y)]dy\nonumber\\
&\leq&\int_{\Sigma_{V_\lambda}^-}G_\infty^\alpha(x^\lambda,y^\lambda)
[v^{p_0}(y)-v_\lambda^{p_0}(y)]dy\nonumber\\
&\leq&p_0\int_{\Sigma_{V_\lambda}^-}G_\infty^\alpha(x^\lambda,y^\lambda)
v^{p_0-1}(y)(-V_\lambda(y))dy.
\label{6.03}
\end{eqnarray*}
By the expression of $G_\infty^\alpha(x,y)$ given in Proposition \ref{pro4.123}, one knows
$$
G_\infty^\alpha(x,y)\leq\frac{A_{n,\alpha}}{|x-y|^{n-\alpha}},
$$
hence 
\begin{equation}
0<u(x)-u_\lambda(x)\leq C\int_{\Sigma_{V_\lambda}^-}\frac{1}{|x-y|^{n-\alpha}}
v^{p_0-1}(y)(-V_\lambda(y))dy.
\label{6.04}
\end{equation}
For any $r>\max\{\frac{n}{n-\alpha},\frac{n}{n-\beta}\}$, apply Proposition \ref{propos4.1230} and H\"{o}lder inequality to (\ref{6.04}), then
\begin{equation}
\|U_\lambda\|_{L^r(\Sigma_{U_\lambda}^-)}\leq C \|v^{p_0-1}V_\lambda\|_{L^{\frac{nr}{n+\alpha r}}(\Sigma_{V_\lambda}^-)}\leq C\|v^{p_0-1}\|_{L^{\frac{n}{\alpha}}(\Sigma_{V_\lambda}^-)}
\|V_\lambda\|_{L^{r}(\Sigma_{V_\lambda}^-)}.
\label{6.05}
\end{equation}
Similarly, we can also obtain
\begin{equation}
\|V_\lambda\|_{L^r(\Sigma_{V_\lambda}^-)}\leq C \|u^{q_0-1}U_\lambda\|_{L^{\frac{nr}{n+\beta r}}(\Sigma_{U_\lambda}^-)}\leq C\|u^{q_0-1}\|_{L^{\frac{n}{\beta}}(\Sigma_{U_\lambda}^-)}
\|U_\lambda\|_{L^{r}(\Sigma_{U_\lambda}^-)}.
\label{6.06}
\end{equation}

(\ref{6.05}) and (\ref{6.06}) imply
\begin{equation}
\|U_\lambda\|_{L^r(\Sigma_{U_\lambda}^-)}\leq C\|v^{p_0-1}\|_{L^{\frac{n}{\alpha}}(\Sigma_{V_\lambda}^-)}
\|u^{q_0-1}\|_{L^{\frac{n}{\beta}}(\Sigma_{U_\lambda}^-)}
\|U_\lambda\|_{L^{r}(\Sigma_{U_\lambda}^-)},
\label{6.07}
\end{equation}
and
\begin{equation}
\|V_\lambda\|_{L^r(\Sigma_{V_\lambda}^-)}\leq C\|v^{p_0-1}\|_{L^{\frac{n}{\alpha}}(\Sigma_{V_\lambda}^-)}
\|u^{q_0-1}\|_{L^{\frac{n}{\beta}}(\Sigma_{U_\lambda}^-)}
\|V_\lambda\|_{L^{r}(\Sigma_{V_\lambda}^-)}.
\label{6.08}
\end{equation}
Since 
$$
\frac{(p_0-1)n}{\alpha}>\frac{n}{n-\beta},~~~\frac{(q_0-1)n}{\beta}
>\frac{n}{n-\alpha},
$$
from (\ref{infi1}) we can deduce that
\begin{equation}
u\in L^{\frac{(q_0-1)n}{\beta}}(\mathbb{R}^n_+),~~~v\in L^{\frac{(p_0-1)n}{\alpha}}(\mathbb{R}^n_+).
\label{6.080}
\end{equation}
Therefore, we can choose $\lambda>0$ sufficiently small such that
$$
C\|v^{p_0-1}\|_{L^{\frac{n}{\alpha}}(\Sigma_{V_\lambda}^-)}
\|u^{q_0-1}\|_{L^{\frac{n}{\beta}}(\Sigma_{U_\lambda}^-)}\leq\frac{1}{2}.
$$
Thus by (\ref{6.05}) and (\ref{6.06}) we conclude that
$$
\|U_\lambda\|_{L^r(\Sigma_{U_\lambda}^-)}=0,~~~\|V_\lambda\|_{L^r(\Sigma_{V_\lambda}^-)}=0,
$$
then it must hold $\Sigma_{U_\lambda}^-=\Sigma_{V_\lambda}^-=\emptyset$, i.e. $\Sigma_\lambda^-=\emptyset$.

\textsl{Step 2.} Now we start form such small positive $\lambda$ and move the plane $\hat{T}_\lambda$ up as long as (\ref{6.01}) holds.

Define 
\begin{equation}
\lambda_0=\sup\{\lambda>0|U_\mu(x),V_\mu(x)\geq0,\forall x\in\Sigma_\mu^*,\mu\leq\lambda\}.
\label{6.09}
\end{equation}
We will prove 
\begin{equation}
\lambda_0=+\infty.
\label{6.10}
\end{equation}

Suppose in the contrary that $\lambda_0<+\infty$, we will show that $u(x),v(x)$ are symmetric about the plane $\hat{T}_{}\lambda_0$, that is 
\begin{equation}
U_{\lambda_0}=V_{\lambda_0}\equiv0,~~~\text{in}~\Sigma_{\lambda_0}^*.
\label{6.11}
\end{equation}
This is a contradiction with $u(x),v(x)>0$ in $\mathbb{R}^n_+$.

Now let us verify (\ref{6.11}) by the contrary. If (\ref{6.11}) is not true, then for such a $\lambda_0$, one has 
\begin{equation}
U_{\lambda_0},V_{\lambda_0}\geq0,~~\text{but}~U_{\lambda_0}\not\equiv0~\text{or}~ V_{\lambda_0}\not\equiv0~~\text{on}~\Sigma_{\lambda_0}^*.
\label{6.110}
\end{equation}
In this case, we are able to prove that the plane can be moved further up. More precisely, we will show that there is a $\varepsilon>0$ such that for all $\lambda\in[\lambda_0,\lambda_0+\varepsilon)$,
\begin{equation}
U_{\lambda},V_{\lambda}\geq0,~~~\text{in}~\Sigma_{\lambda}^*.
\label{6.12}
\end{equation}
If (\ref{6.12}) does not hold, then $\Sigma_{\lambda}^-\neq\emptyset$, we again resort to inequalities (\ref{6.07}) and (\ref{6.08}), and if we can obtain 
\begin{equation}
C\left\{\int_{\Sigma^-_{V_\lambda}}v^{\frac{n(p_0-1)}{\alpha}}(y)dy\right\}
^{\frac{\alpha}{n}}\left\{\int_{\Sigma^-_{U_\lambda}}u^{\frac{n(q_0-1)}{\beta}}(y)dy\right\}
^{\frac{\beta}{n}}\leq\frac{1}{2},
\label{6.13}
\end{equation}
then by (\ref{6.07}) and (\ref{6.08}) we get 
$$
\|U_\lambda\|_{L^r(\Sigma_{U_\lambda}^-)}=
\|V_\lambda\|_{L^r(\Sigma_{V_\lambda}^-)}=0,
$$
which indicates that $\Sigma_{U_\lambda}^-=\Sigma_{V_\lambda}^-=\emptyset$. Hence, for this values of $\lambda\geq\lambda_0$, we have (\ref{6.12}). This contradicts the definition of $\lambda_0$,. Therefore (\ref{6.11}) must be valid.

We postpone the proof of (\ref{6.13}) for a while.

By (\ref{6.11}), we deduce that $u(x)=v(x)=0$ on the plane $x_n=2\lambda_0$, the symmetric image of the boundary $\partial\mathbb{R}^n_+$ with respect to the plane $\hat{T}_{\lambda_0}$, this contradicts our assumption $u,v>0$ in $\mathbb{R}^n_+$. Therefore (\ref{6.10}) must be true. However (\ref{6.10}) implies that the positive solutions $u,v$ are monotone increasing about $x_n$, which again contradicts (\ref{infi1}). Hence (\ref{A4.7}) possesses no positive solutions, and thus (\ref{0}) has no positive solutions either.

Now what left is to show (\ref{6.13}). Since (\ref{infi1}) and (\ref{6.080}) hold, for any small $\eta>0$, we can choose $R$ large enough such that
\begin{equation}
C\left\{\int_{\mathbb{R}^n_+\setminus B_R(0)}v^{\frac{n(p_0-1)}{\alpha}}(y)dy\right\}
^{\frac{\alpha}{n}}\left\{\int_{{\mathbb{R}^n_+\setminus B_R(0)}}u^{\frac{n(q_0-1)}{\beta}}(y)dy\right\}
^{\frac{\beta}{n}}\leq\eta.
\label{6.14}
\end{equation}
Fix this $R$, next we prove that the measure of $\Sigma_\lambda^-\cap B_R(0)$ is sufficiently small for $\lambda$ close to $\lambda_0$. First we show that 
\begin{equation}
U_{\lambda_0}(x),V_{\lambda_0}(x)>0~~~\text{in the interior of}~\Sigma_{\lambda_0}^*.
\label{6.15}
\end{equation}
Indeed, by (\ref{eqw4.201}), (\ref{6.110}) and Proposition \ref{propo4.1230} (iii), we can easily get
\begin{eqnarray}
U_{\lambda_0}(x)&\geq&\int_{\Sigma_{\lambda_0}^*}\left[G_\infty^\alpha
(x^{\lambda_0},y^{\lambda_0})
-G_\infty^\alpha(x,y^{\lambda_0})\right][v_{\lambda_0}^{p_0}(y)-v^{p_0}(y)]dy\nonumber\\
&&+
\int_{\Sigma_{\lambda_0}^C\setminus\tilde{\Sigma}_{\lambda_0}^*}
\left[G_\infty^\alpha(x^{\lambda_0},y)
-G_\infty^\alpha(x,y)\right]v^{p_0}(y)dy\nonumber\\
&\geq&
\int_{\Sigma_{\lambda_0}^C\setminus\tilde{\Sigma}_{\lambda_0}^*}
\left[G_\infty^\alpha(x^{\lambda_0},y)
-G_\infty^\alpha(x,y)\right]v^{p_0}(y)dy\nonumber\\
&>&0.
\label{eqw4.601}
\end{eqnarray}
Similarly, we can also derive $U_{\lambda_0}(x)>0$, and this, together with (\ref{eqw4.601}), verifies (\ref{6.15}).

Since $u,v\in C(\mathbb{R}^n)$, for any given small $\delta>0$, one has by (\ref{6.15}) that
\begin{equation}
U_{\lambda_0}(x),V_{\lambda_0}(x)\geq C_0~~~\text{in}~(\Sigma_{\lambda_0}^*\setminus\Sigma_{\lambda_0-\delta}^*)\cap B_R(0),
\label{6.16}
\end{equation}
where $C_0>0$ is a constant, thus 
\begin{equation}
U_{\lambda}(x),V_{\lambda}(x)\geq C_0~~~\text{in}~(\Sigma_{\lambda_0}^*\setminus\Sigma_{\lambda_0-\delta}^*)\cap B_R(0).
\label{6.17}
\end{equation}
Therefore,
$$
(\Sigma_U^-\cap B_R(0))\subset\Sigma_{\lambda}^*\setminus\Sigma_{\lambda_0-\delta}^*,~~~
(\Sigma_V^-\cap B_R(0))\subset\Sigma_{\lambda}^*\setminus\Sigma_{\lambda_0-\delta}^*.
$$
So, for any $\lambda\in[\lambda_0,\lambda_0+\varepsilon)$, we can let $\varepsilon,\delta>0$ small enough such that $\Sigma_\lambda^-\cap B_R(0)$ is sufficiently small to guarantee
\begin{equation}
C\left\{\int_{\Sigma_{V_\lambda}^-\cap B_R(0)}v^{\frac{n(p_0-1)}{\alpha}}(y)dy\right\}
^{\frac{\alpha}{n}}\left\{\int_{\Sigma_{V_\lambda}^-\cap B_R(0)}u^{\frac{n(q_0-1)}{\beta}}(y)dy\right\}
^{\frac{\beta}{n}}\leq\frac{1}{4}.
\label{6.18}
\end{equation} 

Choose $\eta<1/4$, combining (\ref{6.14}) and (\ref{6.18}), we conclude that (\ref{6.13}) holds for any $\lambda\in(\lambda_0,\lambda_0+\varepsilon)$.

The proof of Claim 3.3 ends here.

\end{proof}

\subsection{Proof of Claim 3.4}

We first list a propositions which will be used in the proof of Claim 3.4.

\begin{proposition} (\cite{CFY})
If $\frac{t}{s}$ is sufficiently small, then for any $x=(x',x_n)$ and $y=(y',y_n)\in R^n_+$,
\begin{equation}
\frac{c_{n,\alpha}}{s^{(n-\alpha)/2}}\frac{t^{\alpha/2}}{s^{\alpha/2}}\leq G^\alpha_\infty(x,y)\leq\frac{C_{n,\alpha}}{s^{(n-\alpha)/2}}\frac{t^{\alpha/2}}{s^{\alpha/2}},
\label{425}
\end{equation}
\begin{equation}
\frac{c_{n,\beta}}{s^{(n-\beta)/2}}\frac{t^{\beta/2}}{s^{\beta/2}}\leq G^\beta_\infty(x,y)\leq\frac{C_{n,\beta}}{s^{(n-\beta)/2}}\frac{t^{\beta/2}}{s^{\beta/2}},
\label{4255}
\end{equation}
that is
\begin{equation}
G^\alpha_\infty(x,y)\sim\frac{t^{\alpha/2}}{s^{n/2}},~~~G^\beta_\infty(x,y)\sim\frac{t^{\beta/2}}{s^{n/2}},
\label{426}
\end{equation}
where $t=4x_ny_n,~s=|x-y|^2$ and $c_{n,\alpha},C_{n,\alpha}$ ($c_{n,\beta},C_{n,\beta}$) stand for different positive constants which only depend on $n,\alpha$ ($n,\beta$).
\label{pro4CFYG}
\end{proposition}

The proof of Proposition \ref{pro4CFYG} is standard and can be found in \cite{CFY}.

Now let us begin to prove Claim 3.4.

\begin{proof}
By Proposition \ref{pro4.123}, to prove Claim 3.4, we only need to show that the positive solutions $u=u(x_n)$ and $v=v(x_n)$ contradict the finiteness of the integrals
$$
\int_{\mathbb{R}^n_+}G^\alpha_\infty(x,y)f(v(y))dy ~~\text{and}~~\int_{\mathbb{R}^n_+}G^\beta_\infty(x,y)g(u(y))dy
$$
respectively.

In fact, by Proposition \ref{pro4CFYG}, if $u(x)=u(x_n)$, $v(x)=v(x_n)$ are a pair of positive solutions to (\ref{424}), then by (\ref{425}), for each fixed $x\in\mathbb{R}^n_+$ and $R$ large enough, we have
\begin{eqnarray}
+\infty>u(x_n)&=&\int_0^\infty f(v(y_n))\int_{\mathbb{R}^{n-1}}G^\alpha_\infty(x,y)dy'dy_n\nonumber\\
&\geq&C\int_R^\infty f(v(y_n))y_n^{\alpha/2}\int_{\mathbb{R}^{n-1}\backslash B_R(0)}\frac{1}{|x-y|^n}dy'dy_n\nonumber\\
&\geq&C\int_R^\infty f(v(y_n)) y_n^{\alpha/2}\int_R^\infty\frac{r^{n-2}}{(r^2+a^2)^\frac{n}{2}}drdy_n\nonumber\\
&\geq&C\int_R^\infty f(v(y_n)) y_n^{\alpha/2}\frac{1}{|x_n-y_n|}\int_{R/a}^\infty\frac{\tau^{n-2}}{(\tau^2+1)^\frac{n}{2}}d\tau dy_n\label{427}\\
&\geq&C\int_R^\infty f(v(y_n))y_n^{\alpha/2-1}dy_n,
\label{429}
\end{eqnarray}
and similarly,
\begin{eqnarray}
+\infty>v(x_n)&=&\int_0^\infty g(u(y_n))\int_{\mathbb{R}^{n-1}}G^\beta_\infty(x,y)dy'dy_n\nonumber\\
&\geq&C\int_R^\infty g(u(y_n))y_n^{\beta/2-1}dy_n,
\label{42888}
\end{eqnarray}
Here to get (\ref{427}), we used $\tau=\frac{r}{a}$.

(\ref{429}) and (\ref{42888}) imply that when $y_n\rightarrow+\infty$,
\begin{equation}
f(v(y_n))y_n^{\alpha/2}\rightarrow0,~~g(u(y_n))y_n^{\beta/2}\rightarrow0.
\label{42909}
\end{equation}
And this indicates that for $y_n$ sufficiently large,
\begin{equation}
f(v(y_n))\sim o((\frac{1}{y_n})^{\alpha/2}),~~g(u(y_n))\sim o((\frac{1}{y_n})^{\beta/2}).
\label{429010}
\end{equation}
Since $f(t),~g(t)\geq0$ are strictly increasing about $t$ in $[0,+\infty)$, we actually have shown that $f(0)=g(0)=0$ and
\begin{equation}
u(y_n)\rightarrow0,~v(y_n)\rightarrow0~~\text{as}~~y_n\rightarrow+\infty.
\label{429011}
\end{equation}
Therefore, by the continuity of $u$ and $v$, there is a positive constant $c$ such that for any $y_n>0$,
\begin{equation}
u(y_n)\leq c,~v(y_n)\leq c.
\label{429012}
\end{equation}

Because $f(t)/t^{p_0}$ is non-increasing about $t$ in $(0,+\infty)$, where $p_0=\frac{n+\alpha}{n-\beta}$, together with (\ref{429012}), it follows that
$$
\frac{f(v(y_n))}{v^{p_0}(y_n)}\geq\frac{f(c)}{c^{p_0}}:=C_0.
$$
Hence, from (\ref{429}) one has
\begin{eqnarray}
+\infty>u(x_n)&\geq&C\int_R^\infty f(v(y_n))y_n^{\alpha/2-1}dy_n\nonumber\\
&\geq&C\int_R^\infty \frac{f(v(y_n))}{v^{p_0}(y_n)}v^{p_0}(y_n)
y_n^{\alpha/2-1}dy_n\nonumber\\
&\geq&C\int_R^\infty v^{p_0}(y_n)y_n^{\alpha/2-1}dy_n,
\label{42929}
\end{eqnarray}
which implies that there exists a sequence $\{y_n^i\}\rightarrow\infty$ as $i\rightarrow\infty$, such that
\begin{equation}
v^{p_0}(y_n^i)(y_n^i)^{\alpha/2}\rightarrow0.
\label{42930}
\end{equation}

Similarly, from (\ref{42888}) and the monotonicity of $\frac{g(t)}{t^{q_0}}$ with $q_0=\frac{n+\beta}{n-\alpha}$, we have
\begin{equation}
+\infty>v(x_n)\geq C\int_R^\infty u^{q_0}(y_n)y_n^{\beta/2-1}dy_n,
\label{429290}
\end{equation}
and there is a sequence $\{y_n^j\}\rightarrow\infty$ as $j\rightarrow\infty$ such that
\begin{equation}
u^{q_0}(y_n^j)(y_n^j)^{\beta/2}\rightarrow0.
\label{429300}
\end{equation}

Next we prove
\begin{equation}
u^{q_0}(y_n)(y_n)^{\beta/2}\geq C>0,
\label{429291}
\end{equation}
for any $y_n>0$, which is a contradiction with (\ref{429300}). Similarly, we can also prove
\begin{equation}
v^{p_0}(y_n)(y_n)^{\alpha/2}\geq C>0,
\label{429292}
\end{equation}
for any $y_n>0$, which is a contradiction with (\ref{42930}). Hence $u=u(x_n)>0$ and $v=v(x_n)>0$ contradict the finiteness of $u$ and $v$ respectively.

Now let us begin to prove (\ref{429291}). Similarly to (\ref{42929}), for any $x=(0,x_n)\in\mathbb{R}^n_+$, we have
\begin{equation}
+\infty>u(x_n)\geq C_0\int_0^\infty v^{p_0}(y_n)y_n^{\alpha/2}\frac{1}{|x_n-y_n|}dy_nx_n^{\alpha/2}.
\label{4i2929}
\end{equation}
Let $x_n=2R$ be sufficiently large, it derives by (\ref{4i2929}) that
\begin{eqnarray}
+\infty>u(x_n)&\geq&C_0\int_0^1 v^{p_0}(y_n)y_n^{\alpha/2}\frac{1}{|x_n-y_n|}dy_nx_n^{\alpha/2}\nonumber\\
&\geq&\frac{C_0}{2R}(2R)^{\alpha/2}\int_0^1 v^{p_0}(y_n)y_n^{\alpha/2}dy_n\nonumber\\
&\geq&C_1(2R)^{\alpha/2-1}=C_1x_n^{\alpha/2-1}.
\label{4ii2929}
\end{eqnarray}
Through the same argument, we can also get
\begin{equation}
+\infty>v(x_n)\geq C_2x_n^{\beta/2-1}.
\label{42i929}
\end{equation}

By (\ref{4i2929}) and (\ref{42i929}), for $x_n=2R$ large enough, we obtain
\begin{eqnarray}
u(x_n)&\geq&C_0\int_{\frac{R}{2}}^R (C_2x_n^{\beta/2-1})^{p_0}y_n^{\alpha/2}\frac{1}{|x_n-y_n|}dy_nx_n^{\alpha/2}\nonumber\\
&\geq&C_0(C_2)^{p_0(\frac{\beta}{2}-1)}(2R)^{p_0(\frac{\beta}{2}-1)}\frac{2}{3R}
(2R)^{\frac{\alpha}{2}}\int_{\frac{R}{2}}^Ry_n^{\frac{\alpha}{2}}dy_n\nonumber\\
&\geq&C_0(C_2)^{p_0(\frac{\beta}{2}-1)}\frac{2^{\alpha/2+2+p_0(\beta/2-1)}}{3(\alpha+2)
(1-\frac{1}{2^{\alpha/2+1}})}R^{p_0(\frac{\beta}{2}-1)+\alpha}\nonumber\\
&:=&AR^{p_0(\frac{\beta}{2}-1)+\alpha}\nonumber\\
&=&\frac{A}{2^{p_0(\frac{\beta}{2}-1)+\alpha}}x_n^{p_0(\frac{\beta}{2}-1)+\alpha}\nonumber\\
&:=&A_1x_n^{p_0(\frac{\beta}{2}-1)+\alpha}.
\label{4i2i929}
\end{eqnarray}
Similarly, we also have
\begin{equation}
v(x_n)\geq B_1x_n^{q_0(\frac{\alpha}{2}-1)+\beta},
\label{4i2i929ii}
\end{equation}
where $A,A_1,B_1$ are positive constants, and in the following, $A_i,B_i, i\geq2, i\in N$ also represent positive constants. Using (\ref{4i2929}) and (\ref{4i2i929ii}), by repeating this procedure one more time, we can derive that for $x_n=2R$,
\begin{equation}
u(x_n)\geq A_2x_n^{p_0q_0(\frac{\alpha}{2}-1)+p_0\beta+\alpha},
\label{04.55}
\end{equation}
and
\begin{equation}
v(x_n)\geq B_2x_n^{p_0q_0(\frac{\beta}{2}-1)+q_0\alpha+\beta}.
\label{04.56}
\end{equation}

\textbf{Case (i).} If we continuing this $2m-1$ times, $m\geq1$, for $x_n=2R$, one gets
\begin{equation}
u(x_n)\geq A_{2m-1}x_n^{\frac{(p_0q_0)^{m-1}-1}{p_0q_0-1}\left[p^2_0q_0(\frac{\beta}{2}-1)+
p_0q_0\alpha+p_0(\frac{\beta}{2}+1)\right]+p_0(\frac{\beta}{2}-1)+\alpha},
\label{429i29}
\end{equation}
and
\begin{equation}
v(x_n)\geq B_{2m-1}x_n^{\frac{(p_0q_0)^{m-1}-1}{p_0q_0-1}\left[p_0q^2_0(\frac{\alpha}{2}-1)+
p_0q_0\beta+p_0(\frac{\alpha}{2}+1)\right]+p_0(\frac{\alpha}{2}-1)+\beta}.
\label{04.58}
\end{equation}
Then
\begin{eqnarray}
u^{q_0}(y_n)(y_n)^{\beta/2}&\geq& A^{q_0}_{2m-1}y_n^{p_0q_0\frac{(p_0q_0)^{m-1}-1}{p_0q_0-1}\left[p_0q_0(\frac{\beta}{2}-1)+
q_0\alpha+\frac{\beta}{2}+1\right]+p_0q_0(\frac{\beta}{2}-1)+q_0\alpha+
\frac{\beta}{2}}\nonumber\\
&:=&A^{q_0}_{2m-1}y_n^{\varphi_{2m-1}(p_0,q_0)},
\label{04.59}
\end{eqnarray}
and
\begin{eqnarray}
v^{p_0}(y_n)(y_n)^{\alpha/2}&\geq& B^{p_0}_{2m-1}y_n^{p_0q_0\frac{(p_0q_0)^{m-1}-1}{p_0q_0-1}\left[p_0q_0(\frac{\alpha}{2}-1)+
p_0\beta+\frac{\alpha}{2}+1\right]+p_0q_0(\frac{\alpha}{2}-1)+p_0\beta+
\frac{\alpha}{2}}\nonumber\\
&:=&B^{p_0}_{2m-1}y_n^{\psi_{2m-1}(p_0,q_0)}.
\label{04.60}
\end{eqnarray}

\textbf{Case (ii).} If we continuing this $2m$ times, $m\geq1$, for $x_n=2R$, one gets
\begin{equation}
u(x_n)\geq A_{2m}x_n^{\frac{(p_0q_0)^{m}-1}{p_0q_0-1}\left[p_0q_0(\frac{\alpha}{2}-1)+
p_0\beta+\frac{\alpha}{2}+1\right]+\frac{\alpha}{2}-1},
\label{004.600}
\end{equation}
and
\begin{equation}
v(x_n)\geq B_{2m}x_n^{\frac{(p_0q_0)^{m}-1}{p_0q_0-1}\left[p_0q_0(\frac{\beta}{2}-1)+
q_0\alpha+\frac{\beta}{2}+1\right]+\frac{\beta}{2}-1}.
\label{04.580}
\end{equation}
Then
\begin{eqnarray}
u^{q_0}(y_n)(y_n)^{\beta/2}&\geq& A^{q_0}_{2m}y_n^{q_0\frac{(p_0q_0)^{m}-1}{p_0q_0-1}\left[p_0q_0(\frac{\alpha}{2}-1)+
p_0\beta+\frac{\alpha}{2}+1\right]+q_0(\frac{\alpha}{2}-1)+
\frac{\beta}{2}}\nonumber\\
&:=&A^{q_0}_{2m}y_n^{\varphi_{2m}(p_0,q_0)},
\label{04.590}
\end{eqnarray}
and
\begin{eqnarray}
v^{p_0}(y_n)(y_n)^{\alpha/2}&\geq& B^{p_0}_{2m}y_n^{p_0\frac{(p_0q_0)^{m}-1}{p_0q_0-1}\left[p_0q_0(\frac{\beta}{2}-1)+
q_0\alpha+\frac{\beta}{2}+1\right]+p_0(\frac{\beta}{2}-1)+
\frac{\alpha}{2}}\nonumber\\
&:=&B^{p_0}_{2m}y_n^{\psi_{2m}(p_0,q_0)}.
\label{04.600}
\end{eqnarray}

Let us first consider \textbf{Case (i).} For any fixed $0<\alpha,\beta<2$, we choose $m$ to be an integer greater than $\frac{n}{(n-2)(\alpha+\beta)}$, i.e.
\begin{equation}
m\geq\left\lceil\frac{n}{(n-2)(\alpha+\beta)}\right\rfloor+1,
\label{04.61}
\end{equation}
where $\lceil a\rfloor$ is the integer part of $a$. We claim that the following inequality is true for such choice of  $m$,
\begin{equation}
\varphi_{2m-1}(p_0,q_0)\geq0.
\label{04.62}
\end{equation}
Actually, since
\begin{equation*}
p_0q_0>1,~~~\frac{(p_0q_0)^{m-1}-1}{p_0q_0-1}>m-1,
\label{04.63}
\end{equation*}
\begin{equation*}
p_0q_0(\frac{\beta}{2}-1)+
q_0\alpha+\frac{\beta}{2}+1=\frac{n(n-2)(\alpha+\beta)}{(n-\alpha)(n-\beta)}>0,
\label{04.633}
\end{equation*}
and
\begin{equation*}
p_0q_0(\frac{\beta}{2}-1)+
q_0\alpha+\frac{\beta}{2}=\frac{n^2(\alpha+\beta-1)-n(\alpha+\beta)-\alpha\beta}
{(n-\alpha)(n-\beta)},
\label{04.6333}
\end{equation*}
then, for any $m$ satisfying (\ref{04.61}), we can verify (\ref{04.62}) through the following elementary computation:
\begin{eqnarray*}
\varphi_{2m-1}(p_0,q_0)&\geq&\frac{(m-1)n(n-2)(\alpha+\beta)+
n^2(\alpha+\beta-1)-n(\alpha+\beta)-\alpha\beta}{(n-\alpha)(n-\beta)}\nonumber\\
&\geq&\frac{(m-1)n(n-2)(\alpha+\beta)+
n^2(\alpha+\beta-1)-n(\alpha+\beta)-n(\alpha+\beta)}{(n-\alpha)(n-\beta)}\nonumber\\
&\geq&\frac{m n(n-2)(\alpha+\beta)-n^2}{(n-\alpha)(n-\beta)}\nonumber\\
&\geq&0.
\label{04.64}
\end{eqnarray*}
Now by (\ref{04.59}) and (\ref{04.62}), we know that (\ref{429291}) is true, which indicates that $v=v(x_n)>0$ contradict the finiteness of $v$.

Similarly, for any fixed $0<\alpha,\beta<2$, we can also choose $m$ to be an integer satisfying (\ref{04.61}) and prove that the following inequality is true for such choice of $m$,
\begin{equation}
\psi_{2m-1}(p_0,q_0)\geq0.
\label{04.65}
\end{equation}
This, together with (\ref{04.60}), verifies that (\ref{429292}) is true, which indicates that $u=u(x_n)>0$ contradict the finiteness of $u$.

Now let us consider \textbf{Case (ii).} For any fixed $0<\alpha,\beta<2$, we choose $m$ to be an integer greater than $\frac{2n-(n-\beta)(\alpha+\beta)}{2(n-2)(\alpha+\beta)}$, i.e.
\begin{equation}
m\geq\left\lceil\frac{2n-(n-\beta)(\alpha+\beta)}{2(n-2)(\alpha+\beta)}\right\rfloor+1.
\label{04.610}
\end{equation}
We claim that the following inequality is true for such choice of  $m$,
\begin{equation}
\varphi_{2m}(p_0,q_0)\geq0.
\label{04.620}
\end{equation}
Actually, since
\begin{equation*}
p_0,q_0>1,~~~\frac{(p_0q_0)^{m}-1}{p_0q_0-1}>m,
\label{04.630}
\end{equation*}
\begin{equation*}
p_0q_0(\frac{\alpha}{2}-1)+
p_0\beta+\frac{\alpha}{2}+1=\frac{2n(n-2)(\alpha+\beta)}{2(n-\alpha)(n-\beta)}>0,
\label{04.6330}
\end{equation*}
and
\begin{equation*}
q_0(\frac{\alpha}{2}-1)+
\frac{\beta}{2}=\frac{n^2(\alpha+\beta-2)-\beta n(\alpha+\beta)+2\beta^2}
{2(n-\alpha)(n-\beta)},
\label{04.63330}
\end{equation*}
then, for any $m$ satisfying (\ref{04.610}), we can verify (\ref{04.620}) through the following elementary computation:
\begin{eqnarray*}
\varphi_{2m}(p_0,q_0)&\geq&\frac{2m n(n-2)(\alpha+\beta)+
n^2(\alpha+\beta-2)-\beta n(\alpha+\beta)+2\beta^2}{2(n-\alpha)(n-\beta)}\nonumber\\
&\geq&\frac{2m n(n-2)(\alpha+\beta)+
n^2(\alpha+\beta-2)-\beta n(\alpha+\beta)}{2(n-\alpha)(n-\beta)}\nonumber\\
&\geq&\frac{n[2(n-2)(\alpha+\beta)m+(n-\beta)(\alpha+\beta)-2n]}{2(n-\alpha)(n-\beta)}\nonumber\\
&\geq&0.
\label{04.640}
\end{eqnarray*}
Now by (\ref{04.590}) and (\ref{04.620}), we know that (\ref{429291}) is true, which indicates that $v=v(x_n)>0$ contradict the finiteness of $v$.

Similarly, for any fixed $0<\alpha,\beta<2$, we can also choose $m$ to be an integer satisfying (\ref{04.610}) and prove that the following inequality is true for such choice of $m$,
\begin{equation}
\psi_{2m}(p_0,q_0)\geq0.
\label{04.650}
\end{equation}
This, together with (\ref{04.600}), verifies that (\ref{429292}) is true, which again indicates that $u=u(x_n)>0$ contradict the finiteness of $u$.

At present, we have shown that $u=u(x_n)$, $v=v(x_n)$ contradict the finiteness of $u,v$ respectively. Hence Claim 3.4 is proved.

\end{proof}




\end{document}